\documentclass[a4paper,10pt,twoside]{article}
\pdfoutput=1
\usepackage[T1]{fontenc} 
\usepackage[utf8]{inputenc}
\usepackage[british]{babel}
\usepackage{mathrsfs}
\usepackage{amsfonts}
\usepackage{amsmath} 
\usepackage{amsthm}
\usepackage{amssymb}
\usepackage{mathtools}
\usepackage{tikz}
\usetikzlibrary{matrix,arrows,decorations.pathmorphing, shapes}
\usepackage{centernot}
\usepackage{indentfirst}
\usepackage[affil-it]{authblk}
\usepackage{hyperref}
\hypersetup{colorlinks=false, pdfborder={0 0 0}}
\usepackage[inline]{enumitem}
\setdescription{font=\normalfont\itshape}     
\usepackage{fancyhdr}

\usepackage{url}
\usepackage{cite}
\usepackage{rotating}
\usepackage{stmaryrd}
\makeatletter                        
\g@addto@macro{\UrlBreaks}{\UrlOrds} 
\makeatother                         
\newsavebox\MBox
\newcommand{\from}{\colon}
\newcommand{\allora}{\Rightarrow}
\newcommand{\sse}{\Leftrightarrow}
\newcommand{\implica}{\rightarrow}
\newcommand{\coimplica}{\leftrightarrow}
\newcommand{\ssq}{\subseteq}
\renewcommand{\phi}{\varphi}
\renewcommand{\epsilon}{\varepsilon}
\renewcommand{\models}{\vDash}

\newcommand{\mb}{\mathbb}
\newcommand{\mf}{\mathfrak}

\newcommand{\tf}{\textsf}

\newcommand{\proves}{\vdash}
\newcommand{\pfin}{{\mathscr P}_{\mathrm{fin}}}
\newcommand{\opsc}[1]{\operatorname{\textsc{#1}}}
\newcommand{\restr}{\upharpoonright}
\newcommand*\bigcircled[1]{\tikz[baseline=(char.base)]{
    \node[shape=circle,draw,inner sep=1pt] (char) {#1};}}
\newcommand{\monster}{\mf U}
\newcommand{\smallsubset}{\mathrel{\subset^+}}

\newcommand{\smallprec}{\mathrel{\prec^+}}
\newcommand{\satext}{\mathrel{^+\!\!\succ}}

\newcommand{\invtypes}{S^{\mathrm{inv}}}
\newcommand{\invext}{\mid}
\newcommand{\then}{\Longrightarrow}
\newcommand{\eqqcolon}{\mathrel{\rotatebox[origin=c]{180}{$\coloneqq$}}}
\newcommand{\ceq}{\coloneqq}
\newcommand{\eq}{\mathrm{eq}}
\def\forkindep{\mathrel{\raise0.2ex\hbox{\ooalign{\hidewidth$\vert$\hidewidth\cr\raise-0.9ex\hbox{$\smile$}}}}}
\newcommand{\find}[1]{\mathchoice{\mathrel{\underset{#1}\forkindep}}{\mathrel{\forkindep_{#1}}}{}{}}
\def\forkdep{\mathrel{\raise0.2ex\hbox{\ooalign{\hidewidth$\not\vert$\hidewidth\cr\raise-0.9ex\hbox{$\smile$}}}}}

\newcommand{\invbar}{\operatorname{\overline{Inv}}(\monster)}
\newcommand{\invtilde}{\operatorname{\widetilde{Inv}}(\monster)}
\newcommand{\gsinvtilde}{\operatorname{\widetilde{Inv}^\mathrm{gs}}\!\!(\monster)}
\newcommand{\equidom}{\mathrel{\equiv_{\mathrm{D}}}}
\newcommand{\nequidom}{\mathrel{{\centernot\equiv}_{\mathrm{D}}}}
\newcommand{\domeq}{\mathrel{\sim_{\mathrm{D}}}}
\newcommand{\doms}{\mathrel{\ge_{\mathrm{D}}}}
\newcommand{\ndoms}{\mathrel{{\centernot\ge}_{\mathrm{D}}}}
\newcommand{\domd}{\mathrel{\le_{\mathrm{D}}}}
\newcommand{\cldomeq}{\mathrel{\bowtie}} 
 
\newcommand{\cldoms}{\mathrel{\triangleright}}\newcommand{\ka}{\kappa}
\newcommand{\wort}{\mathrel{\perp^\mathrm{w}}}
\newcommand{\nwort}{\mathrel{\not\perp^\mathrm{w}}}
\newcommand{\inverse}{^{-1}}
\newcommand{\bla}[4]{{#1}_{#2}#3\ldots#3{#1}_{#4}}
\newcommand{\pow}[2]{#1^{(#2)}}
\newcommand{\actson}{\mathrel{\curvearrowright}}
\DeclareMathOperator{\im}{Im}
\DeclareMathOperator{\Th}{Th}

\DeclareMathOperator{\tp}{tp}
\DeclareMathOperator{\dcl}{dcl}
\DeclareMathOperator{\acl}{acl}
\DeclareMathOperator{\aut}{Aut}
\DeclarePairedDelimiter{\set}{\{}{\}}
\DeclarePairedDelimiter{\abs}{\lvert}{\rvert}
\DeclarePairedDelimiter{\seq}{\langle}{\rangle}
\DeclarePairedDelimiter{\class}{\llbracket}{\rrbracket}
\theoremstyle{definition}
\newtheorem{defin}{Definition}[section]
\newtheorem{thm}[defin]{Theorem}
\newtheorem*{unnmdthm}{Theorem}
\newtheorem{pr}[defin]{Proposition}
\newtheorem{co}[defin]{Corollary}
\newtheorem{lemma}[defin]{Lemma}
\newtheorem*{notation}{Notation}
\newtheorem{eg}[defin]{Example}
\newtheorem{rem}[defin]{Remark}
\newtheorem{fact}[defin]{Fact}
\newtheorem{question}[defin]{Question}
\newtheorem{conjecture}[defin]{Conjecture}

\makeatletter
\newcommand{\subjclass}[2][2010]{%
  \let\@oldtitle\@title%
  \gdef\@title{\@oldtitle\footnotetext{#1 \emph{Mathematics subject classification.} #2}}%
}
\newcommand{\keywords}[1]{%
  \let\@@oldtitle\@title%
  \gdef\@title{\@@oldtitle\footnotetext{\emph{Key words and phrases.} #1.}}%
}
\makeatother

\author{Rosario Mennuni%
  \thanks{email: \url{R.Mennuni@leeds.ac.uk} \textsc{orcid}: \url{https://orcid.org/0000-0003-2282-680X}}}
\affil{University of Leeds}
\title{Product of Invariant Types\\ Modulo Domination-Equivalence}

\begin{document}
\maketitle
\begin{abstract}
    We investigate the interaction between the  product of invariant types and domination-equivalence. We present a theory where the latter is not a congruence with respect to the former, provide sufficient conditions for it to be, and study the resulting quotient when  it is.
\end{abstract}

To a sufficiently saturated model of a first-order theory one can associate a semigroup, that of \emph{global invariant types}  with the \emph{tensor product} $\otimes$.  This can be endowed with two equivalence relations, called \emph{domination-equivalence} and \emph{equidominance}. This paper  studies the resulting quotients, starting from sufficient conditions for $\otimes$ to be well-defined on them. We show, correcting a remark in~\cite{hhm}, that this need not be always the case.

Let $S(\monster)$ be the space of types in any finite number of variables over a model $\monster$ of a first-order theory that is $\kappa$-saturated and $\kappa$-strongly homogeneous for some large $\kappa$. For any set $A\subseteq \monster$, one has a natural action on $S(\monster)$ by the group $\aut(\monster/A)$ of automorphisms of $\monster$ that fix $A$ pointwise.  The space $\invtypes(\monster)$ of \emph{global invariant types} consists of those elements of $S(\monster)$ which, for some \emph{small} $A$, are fixed points of the  action $\aut(\monster/A)\actson S(\monster)$. Each of these types has a canonical extension to bigger models $\monster_1\succ \monster$, namely the unique one which is a fixed point of the action  $\aut(\monster_1/A)\actson S(\monster_1)$, and this allows us to define an associative product $\otimes$ on $\invtypes(\monster)$. This is the semigroup which we are going to quotient.

We say that a global type $p$ \emph{dominates} a global type $q$ when $p$ together with a small set of formulas entails $q$. This is a preorder, and we call the induced equivalence relation \emph{domination-equivalence}. We also look at \emph{equidominance}, the refinement of domination-equivalence obtained by requiring that domination of $p$ by $q$ and of $q$ by $p$ can be witnessed by the same set of formulas. These notions have their roots in the work of Lascar, who in~\cite{firstdefinition} generalised the Rudin-Keisler order on ultrafilters to types of a theory; his preorder was subsequently generalised to domination between stationary types in a stable theory. 

Equidominance reached its current form in~\cite{hhm}, where it was used to prove a result of Ax-Kochen-Ershov flavour; namely, that in the case of algebraically closed valued fields one can compute the quotient of the semigroup of global invariant types by equidominance, and it turns out to be commutative and to decompose in terms of value group and residue field. It was also claimed, without proof, that such a semigroup is also well-defined and commutative in any complete first-order theory. The starting point of this research was to try to fill this gap by proving these claims. After trying in vain to prove well-definedness of the quotient semigroup, the author started to investigate sufficient conditions for it to hold. Eventually, a counterexample arose:
\begin{unnmdthm}
  There is a ternary, $\omega$-categorical, supersimple theory of SU-rank $2$ with degenerate algebraic closure in which neither  domination-equivalence nor equidominance are congruences with respect to $\otimes$.
\end{unnmdthm}

The paper is organised as follows. In Section~\ref{section:defwd} we define the main object of study, namely the quotient  $\invtilde$ of the semigroup of global invariant types modulo domination-equivalence, provide some sufficient conditions for it to be well-defined and investigate its most basic properties. In Section~\ref{section:cntrex}   we prove the theorem above, which shows that $\invtilde$ need not be well-defined in general; we also  show (Corollary~\ref{co:rgncomm}) that in the theory of the Random Graph $\invtilde$ is not commutative. In Section~\ref{section:prpr} we prove that definability, finite satisfiability, generic stability (Theorem~\ref{thm:propertiespreserved}) and weak orthogonality  to a type (Proposition~\ref{pr:wortpreserved}) are preserved downwards by domination. This is useful in explicit computations of $\invtilde$ and yields as a by-product (Corollary~\ref{co:gsinvtilde}) that another, smaller object based on generically stable types is instead well-defined in full generality. In    Section~\ref{section:depmon} we explore whether and how much $\invtilde$ depends on $\monster$; we show (Corollary~\ref{co:ipdep}) that its independence from the choice of $\monster$ implies \textsf{NIP}. Section~\ref{section:stabletheories}  gathers some previously known results from classical stability theory and explores their consequences in the context of this paper (e.g.~Theorem~\ref{thm:bigoplusN}). Sections from~\ref{section:cntrex} to~\ref{section:depmon} depend on Section~\ref{section:defwd} but can be read independently of each other; Section~\ref{section:stabletheories} contains references to all previous sections but can in principle be read after Section~\ref{section:defwd}. 

 \paragraph{Acknowledgements}
 First of all, I would like to thank my supervisors, Dugald Macpherson and Vincenzo Mantova, whose guidance, suggestions and feedback have been more than invaluable. I am also indebted to Jan Dobrowolski for the useful discussions on binary theories,  Ehud Hrushovski for his permission to include Proposition~\ref{pr:dlopncomm}, and Anand Pillay for pointing out Remark~\ref{rem:tteqstab}.

 Part of this research was carried out while I was participating in the thematic trimester \emph{Model Theory, Combinatorics and Valued Fields} held in Paris from January to April 2018. I would like to thank the organising committee for making the trimester possible, and the Institut Henri Poincar\'e, the Centre National de la Recherche Scientifique and the School of Mathematics of the University of Leeds for supporting financially my participation in it.
 
 This research is part of the author's PhD project, supported by a Leeds Anniversary Research Scholarship.

\section{Definition and Well-Definedness}\label{section:defwd}
\subsection{Set-up}
Notations and conventions are standard, and we now recall some of them.

We work in an arbitrary complete theory $T$, in a first-order language $L$, with infinite models. As customary,  all mentioned inclusions between models of $T$ are assumed to be elementary maps, and we call models of $T$ which are $\kappa$-saturated and $\kappa$-strongly homogeneous for a large enough $\kappa$   ``monster'' models; we denote them by $\monster$, $\monster_0$, etc.  Saying that $A\subseteq \monster$ is \emph{small} means that $\monster$ is $\abs A^+$-saturated and $\abs A^+$-strongly homogeneous, and is sometimes denoted by $A\smallsubset \monster$, or $A\smallprec \monster$ if additionally $A\prec \monster$. \emph{Large} means ``not small''. The letters $A$ and $M$ usually represent, respectively, a small subset and a small elementary substructure  of $\monster$.

Parameters and variables are tacitly allowed to be finite tuples unless otherwise specified, and we abuse the notation by writing e.g.\ $a\in \monster$ instead of $a\in \monster^{\abs a}$. Coordinates of a tuple are indicated with subscripts, starting with $0$, so for instance $a=(a_0,\ldots, a_{\abs a-1})$. To avoid confusion, indices for a sequence of tuples are written as superscripts, as in $\seq{a^i\mid i\in I}$. The letters $x,y,z,w,t$ denote tuples of variables, the letters $a,b,c,d,e,m$ denote tuples of elements of a model.

A \emph{global type} is a complete type over $\monster$.  ``Type over $B$'' means ``complete type over $B$''. We say ``partial type'' otherwise. We sometimes write e.g.\ $p_x$ in place of $p(x)$ and denote with $S_x(B)$ the space of types in variables $x$.

When mentioning realisations of global types, or supersets of a monster, we implicitly think of them as living inside a bigger monster model, which usually goes unnamed. Similarly, implications are to be understood modulo the elementary diagram $\opsc{ed}(\monster_*)$ of an ambient monster model $\monster_*$, e.g.\ if $c\in \monster_*\succ \monster$ and $p\in S(\monster c)$ then $(p\restr \monster)\proves p$ is a shorthand for $(p\restr \monster)\cup \opsc{ed}(\monster_*)\proves p$. We sometimes take deductive closures implicitly, as in ``$\set{x=a}\in S_x(\monster)$''.

  If we define a property a theory may have, and then we say that a structure has it, we mean that its complete theory does. When we say ``$L$-formula'', we mean without parameters;  for emphasis, we sometimes write $L(\emptyset)$, with the same meaning as $L$. In formulas, (tuples of) variables will be separated by commas or semicolons. The distinction is purely cosmetic, to help readability, and usually it means we regard the variables on the left of the semicolon as ``object variables'' and the ones on the right as ``parameter variables'', e.g.\ we may write $\phi(x,y;w)\in L$, $\phi(x,y;d)\in p(x)\otimes q(y)$.

\subsubsection{Products of Invariant Types}
We briefly recall some standard results on invariant types and fix some notation. For proofs, see e.g.~\cite[Section~2.2]{simon} or~\cite[Chapter~12]{poizat}.

\begin{defin}
Let $A\subseteq B$.  A  type $p\in S_x(B)$ is \emph{$A$-invariant} iff for all $\phi(x;y)\in L$ and $a\equiv_A b$ in $B^{\abs y}$ we have $p(x)\proves \phi(x;a)\coimplica \phi(x;b)$. A global type $p\in S_x(\monster)$ is \emph{invariant} iff it is $A$-invariant for some small $A$.
\end{defin}
Equivalently, a global $p\in S_x(\monster)$ is $A$-invariant iff it is a fixed point of the usual action of $\aut(\monster/A)$ on $S_x(\monster)$ defined by $f(p)\coloneqq\set{\phi(x; f(a))\mid \phi(x; y)\in L(\emptyset), \phi(x;a)\in p}$. Note that if $p$ is $A$-invariant and $A_1\supseteq A$, then $p$ is automatically $A_1$-invariant. This will be used tacitly throughout.
\begin{notation}
We denote by $\invtypes_x(\monster, A)$ the space of global $A$-invariant types in variables $x$, with $A$ small, and with $\invtypes_x(\monster)$ the union of the $\invtypes_x(\monster, A)$ as $A$ varies among small subsets of $\monster$. We denote by $S_{<\omega}(B)$, or just by $S(B)$, the union  for $n<\omega$ of the spaces of complete types in $n$ variables over $B$. Similarly for, say,  $\invtypes_{<\omega}(\monster)$.

If $p\in S_x(\monster)$ is $A$-invariant and $\phi(x;y)\in L$, write 
\[
(d_p\phi(x;y))(y)\coloneqq\set{\tp_y(b/A)\mid \phi(x;b)\in p, b\in \monster}
\]

If $p(x),q(y)\in S(B)$ and $A\subseteq B$, we write
\[
  S_{pq}(A)\coloneqq\set{r\in S_{xy}(A)\mid r\supseteq (p\restr A)\cup (q\restr A)}
\]
In situations like the one above, we implicitly assume, for convenience and with no loss of generality, that $x$ and $y$ share no common variable.
\end{notation}

\begin{pr}[{\!\!\cite[p.~19]{simon}}]
Let $A$ be small.  Given an $A$-invariant type $p\in S_x(\monster)$ and a set of parameters $B\supseteq \monster$ there is a unique extension $p\invext B$ of $p$ to an $A$-invariant type over $B$, and it is given by requiring, for all $\phi(x;y)\in L$ and $b\in B$, 
\[
\phi(x;b)\in p\invext B\iff \tp(b/A)\in (d_p\phi(x;y))(y)
\]
Moreover, if $p\in \invtypes_x(\monster, A)$,  $\phi(x,y;w)\in L(\emptyset)$,  $d\in \monster$ and $q\in S_y(\monster)$, then the following are equivalent:
\begin{enumerate}
\item For some (equivalently, all) $b\models q$ we have that $\phi(x,b;d)\in p\invext \monster b$.
\item For some  (equivalently, all) $b\in\monster$ such that $b\models q\restr Ad$ we have that $\phi(x,b;d)\in p$.
\item $q\in \pi^{-1}\bigl((d_p\phi(x,y;d))(y)\bigr)$, for $\pi\from S_y(\monster)\to S_y(Ad)$  the restriction map.
\end{enumerate}
\end{pr}
Also note that if $A_1\supseteq A$ is another small set then $p\invext B$ is also the unique $A_1$-invariant extension of $p$. All this ensures that the following operation is well-defined, i.e.\ does not depend on $b\models q$ and on whether we regard $p$ as $A$-invariant or $A_1$-invariant.
\begin{defin}
Let $p\in \invtypes_x(\monster,A)$  and $q\in S_y(\monster)$. Define $p(x)\otimes q(y)\in S_{xy}(\monster)$ as follows. Fix $b\models q$. For each $\phi(x,y)\in L(\monster)$, define
\[
\phi(x,y)\in p(x)\otimes q(y)\iff \phi(x,b)\in p\invext \monster b
\]
We also define inductively $\pow p1\coloneqq p(x^0)$ and $\pow{p}{n+1}\coloneqq p(x^n)\otimes\pow pn(x^{n-1},\ldots, x^0)$.
\end{defin}

\begin{fact}[\!\!{\cite[Fact~2.19 and Fact~2.20]{simon}}]
The product of two $A$-invariant global types is still $A$-invariant, and $\otimes$ is associative on $\invtypes(\monster)$.
\end{fact}
  \begin{eg}
    If $T$ is stable then $\invtypes(\monster)=S(\monster)$ and $p\otimes q=\tp(a,b/\monster)$ where $a\models p$, $b\models q$ and $a\find \monster b$. If $T=\textsf{DLO}$ and $p(x)=\tp(+\infty/\monster)$, then $p(x)\otimes p(y)=p(x)\cup p(y)\cup\set{x>y}$.
  \end{eg}
  \subsubsection{Domination}
\begin{defin} \label{defin:domination}
Let $p\in S_x(\monster)$ and $q\in S_y(\monster)$.
\begin{enumerate}
\item We say that $p$ \emph{dominates} $q$ and write $p\doms q$ iff there are some small $A$ and some $r\in S_{xy}(A)$ such that
\begin{itemize}
\item $r\in S_{pq}(A)$, and
\item $p(x)\cup r(x,y)\proves q(y)$.
\end{itemize}
\item We say that $p$ and $q$ are \emph{domination-equivalent} and write $p\domeq q$ iff $p\doms q$ and $q\doms p$.
\item We say that $p$ and $q$ are \emph{equidominant} and write $p\equidom q$ iff  there are some small $A$ and some $r\in S_{xy}(A)$ such that
\begin{itemize}
\item $r\in S_{pq}(A)$,
\item $p(x)\cup r(x,y)\proves q(y)$, and
\item $q(y)\cup r(x,y)\proves p(x)$.
\end{itemize}
\end{enumerate}
\end{defin}
So $p\equidom q$ if and only if both $p\doms q$ and $q\doms p$ hold, and  both statements can be witnessed by the same $r$.
To put it differently, a direct definition of   $p\domeq q$ can be obtained by replacing, in the last clause of the definition of $p\equidom q$, the small type $r$ with another small type $r'$, possibly different from $r$. That the last two relations are in general distinct can be seen for instance in \textsf{DLO} together with a dense-codense predicate; see Example~\ref{eg:deq}.

Note that we are not requiring $p\cup r$ to be complete; in other words, domination is ``small-type semi-isolation'', as opposed to ``small-type isolation''. The finer relation of  \emph{semi-isolation}, also known as the \emph{global RK-order},\footnote{Strictly speaking, the original definition of the RK-order in~\cite[D\'efinition~1]{firstdefinition}  slightly differs from the relation that customarily bears the same name in the literature.} was studied  for instance in~\cite{tanovic}.
\begin{pr}
$\doms$ and $\equidom$ are respectively a preorder and an equivalence relation on $S_{<\omega}(\monster)$. Consequently, $\domeq$ is an equivalence relation as well.
\end{pr}
\begin{proof}
  The only non-obvious thing is transitivity. We prove it for $\equidom$ first, as the proof for $\doms$ is even easier. Suppose that $r(x,y)\in S_{p_0p_1}(A_r)$ witnesses that $p_0(x)\equidom p_1(y)$ and that $s(y,z)\in S_{p_1p_2}(A_s)$ witnesses $p_1(y)\equidom p_2(z)$. Up to taking a larger $A$ and then completing $r$, $s$ to types with parameters from $A$, we can assume $A_r=A_s=A$. By hypothesis and compactness, for every formula $\phi(z)\in p_2$ there are formulas $\psi(y,z)\in s$, $\theta(y)\in p_1$ and $\chi(x,y)\in r$ such that $p_0\cup \set{\chi(x,y)}\proves \theta(y)$ and $\set{\theta(y)\land \psi(y,z)}\proves \phi(z)$. If we let  $\sigma_\phi(x,z)\coloneqq \exists y\; (\chi(x,y)\land \psi(y,z))$, then  $p_0(x)\cup \set{\sigma_\phi(x,z)}\proves \phi(z)$. Moreover,  we have $\sigma_\phi(x,z)\in L(A)$. Analogously, for each $\delta(x)\in p_0$ we can find $\rho_\delta(z,x)\in L(A)$ such that $p_2(z)\cup \set{\rho_\delta(z,x)}\proves\delta(x)$, obtained in the same way \emph{mutatis mutandis}. It is now enough to show that the set \[\Phi\coloneqq p_0(x)\cup r(x,y)\cup p_1(y)\cup s(y,z)\cup p_2(z)\] is consistent, as this will in particular entail consistency of  \[\set{\sigma_\phi\mid \phi\in p_2}\cup \set{\rho_\delta\mid \delta\in p_0}\cup(p_0\restr A)\cup (p_2\restr A)\] which will therefore have a completion to a type in $S_{p_0p_2}(A)$  witnessing $p_0\equidom p_2$. To see that $\Phi$ is consistent, in a larger monster $\monster_1$ let $(a,b)\models p_0\cup r$ and $(\tilde b, \tilde c)\models p_1\cup s$. Since $\tp(b/\monster)=p_1=\tp(\tilde b/\monster)$, there is $f\in \aut(\monster_1/\monster)$ such that $f(\tilde b)=b$, and  then $(a,b,f(\tilde c))\models\Phi$.

The proof for  $\doms$ is exactly the same, except we do not need to consider the $\rho_\delta$ formulas.
\end{proof}
As we are interested in the interaction of these notions with $\otimes$, we restrict our attention to  quotients of $\invtypes(\monster)$. Note that, by the following lemma, whether or not $p\in \invtypes(\monster)$ only depends on its equivalence class.

\begin{lemma}\label{lemma:invariancepreserved}
  If $p\in \invtypes_x(\monster, A)$ and $r\in S_{xy}(B)$ are such that $p\cup r$ is consistent and  $p\cup r\proves q\in S_y(\monster)$, then $q$ is invariant over $AB$.
\end{lemma}
\begin{proof}
The set of formulas $p\cup r$ is fixed by $\aut(\monster/AB)$ and implies $q$. As $q$ is complete, the conclusion follows.
\end{proof}
Anyway, $q$ will not be in general $A$-invariant: for instance, by the proof of point~\ref{point:min} of Proposition~\ref{pr:realneut}, for every $p$ and every realised $q$ we have $p\doms q$, and it is enough to take $q$ realised in $\monster\setminus \dcl(A)$ to get a counterexample.

\begin{defin}
  Let $\invtilde$ be the quotient of $\invtypes(\monster)$ by $\domeq$, and $\invbar$ the quotient of $\invtypes(\monster)$ by $\equidom$.
\end{defin}
Note that, if $p\cup r\proves q$, by passing to a suitable extension of $r$ there is no harm in enlarging its domain, provided it stays small.  This sort of manipulation will from now on be done tacitly. 
\begin{rem}\label{rem:changeinterminology}
  In~\cite{hhm}, the name \emph{domination-equivalence} is used to refer to $\equidom$  (no mention is made of $\doms$ and $\domeq$). The reason for this change in  terminology is to ensure consistency with the notions with the same names classically defined for stable theories, which coincide with the ones just defined (see Section~\ref{section:stabletheories}). As $\invtilde$  carries a poset structure, and  is in some sense better behaved than $\invbar$, we mostly focus on the former. 
\end{rem}
\begin{eg}\label{eg:deq}\*
  \begin{enumerate}
  \item It is easy to see that, in any strongly minimal theory, two global types are domination-equivalent, equivalently equidominant, precisely when they have the same dimension over $\monster$. 
  \item In \textsf{DLO}, if $p(x)$ is the type at $+\infty$, then  $p(x)\equidom p(y)\otimes p(z)$, as can be easily seen by using some $r$ containing the formula $x=z$. 
  \item The two equivalence relations differ in the theory \textsf{DLOP} of a \textsf{DLO} with a dense-codense predicate $P$. In this case, if $p(x)$ is the type at $+\infty$ in $P$, and $q(y)$ is the type at $+\infty$ in $\neg P$, then $p(x)\doms q(y)$ (resp.~$p(x)\domd q(y)$) can be witnessed by any $r$ containing $y>x$ (resp.~$y<x$). To show $p\nequidom q$,  take any $r\in S_{pq}(A)$. Since $(p(x)\restr \emptyset)\cup(q(y)\restr \emptyset)\proves P(x)\land \neg P(y)$ we have $r\proves x\ne y$, and since $A$ is small there is $b\in \monster$ such that $b>A$. It  follows from quantifier elimination that, if for instance $r\proves x>y$, then $p\cup r\centernot \proves y>b$, and a fortiori $p\cup r\centernot \proves q$.  The reason the two equivalence relations may differ is, simply, that even if there are $r_0$ and $r_1$ such that  $p\cup r_0\proves q$ and $q\cup r_1\proves p$, we may still have that the union $r_0\cup r_1$ is inconsistent.
  \item  The two equivalence relations may differ  even in a stable theory, as shown by~\cite[Example~5.2.9]{wagner} together with the fact (Proposition~\ref{pr:sameasclassic}) that the classical definitions via forking (see Definition~\ref{defin:cldomtypes}) in stable theories coincide with the ones in Definition~\ref{defin:domination}.
  \end{enumerate}
\end{eg}

\subsubsection{Interaction with $\otimes$}
We start our investigation of the compatibility of $\otimes$ with $\doms$ and $\equidom$ with two easy lemmas. While the first one will not be needed until later, the second one will be used repeatedly.
\begin{lemma}\label{lemma:babycraig}
  If $A\ssq B\ssq C$, $p_x,q_y\in S(C)$ and $r\in S_{pq}(A)$ is such that  $p\cup r\proves q$, then $(p\restr B)\cup r\proves q\restr B$.
\end{lemma}
\begin{proof}
Let $\psi(y) \in q\restr B$. By hypothesis and compactness there is $\chi(x,y)\in r$ such that $p \proves \forall y\; (\chi(x,y)\implica\psi(y))$.
As $A\ssq B$,  this formula is in $p\restr B$.
\end{proof}

\begin{lemma}\label{lemma:rsufficesforinvext}
If $p_x,q_y\in\invtypes(\monster, A)$ and $r\in S_{pq}(A)$ is such that $p\cup r\proves q$, then for all sets of parameters $B\supseteq \monster$ we have  $(p\invext B)\cup r\proves q\invext B$.
\end{lemma}
\begin{proof}
  Let $\phi(y;w)$ be an $L(\emptyset)$-formula and $b\in B$ be such that $\phi(y;b)\in q\invext B$. Pick any $\tilde b\in \monster$ such that $\tilde b\equiv_A b$. By definition of $q\invext B$ we have $\phi(y;\tilde b)\in q$, so by hypothesis and compactness there is an $L(A)$-formula $\psi(x,y)\in r(x,y)$ such that $p\proves \forall y\; \bigl(\psi(x,y)\implica \phi(y;\tilde b)\bigr)$.
But then, by definition of $p\invext B$ and the fact that $\psi\in L(A)$ we have $p\invext B\proves \forall y\; \bigl(\psi(x,y)\implica \phi(y; b)\bigr)$, and since $\psi\in r$ we get $(p\invext B)\cup r\proves \phi(y;b)$.
\end{proof}

\begin{notation}
We adopt from now on the following conventions. The letter $A$ continues to denote a small set. The symbols $p$, $q$, possibly with subscripts, denote global $A$-invariant types, and $r$ stands for an element of, say, $S_{pq}(A)$ witnessing  domination or equidominance.
\end{notation}
The first use we make of Lemma~\ref{lemma:rsufficesforinvext} is to prove the following statement, which generalises~\cite[Corollaire~11]{firstdefinition}.
\begin{lemma}\label{lemma:atleastontheleft}
  Suppose $p_0(x)\cup r(x,y)\proves p_1(y)$, and let $s\coloneqq r(x,y)\cup \set{z=w}$.
  Then $(p_0(x)\otimes q(z))\cup s\proves p_1(y)\otimes q(w)$.
  In particular if $p_0\doms p_1$ then $p_0\otimes q\doms p_1\otimes q$, and the same holds replacing $\doms$ with $\equidom$.
\end{lemma}
\begin{proof}
Choose any $b\models q(z)$. For any  $\phi(y,z;t)\in L(\emptyset)$ and $d\in \monster$  such that $\phi(y,z;d)\in p_1(y)\otimes q(z)$  we have, by definition of $\otimes$, that $p_1(y)\invext \monster b\models \phi(y, b; d)$.
By Lemma~\ref{lemma:rsufficesforinvext} there is some $L(A)$-formula $\psi(x,y)\in r(x,y)$ such that $p_0(x)\invext \monster b\models \forall y\; \bigl(\psi(x,y)\implica\phi(y, b; d)\bigr)$, hence $p_0(x)\otimes q(z)\models  \forall y\;  \bigl(\psi(x,y)\implica\phi(y, z; d)\bigr)$.
In particular, since $\psi\in r$, we have $(p_0(x)\otimes q(z))\cup r\proves \phi(y,z;d)$. Therefore any completion of $s\cup \bigl((p_0(x)\otimes q(z))\restr A\bigr)\cup \bigl((p_1(y)\otimes q(w))\restr A\bigr)$ witnesses that  $p_0(x)\otimes q(z)\doms p_1(y)\otimes q(w)$.

In the special case where  the same $r$ also witnesses $p_1\doms p_0$, for same $s$ we have  that $s\cup \bigl((p_0(x)\otimes q(z))\restr A\bigr)\cup \bigl((p_1(y)\otimes q(w))\restr A\bigr)$ witnesses $p_1\otimes q\doms p_0\otimes q$, and we get $p_1\otimes q\equidom p_0\otimes q$.
\end{proof}
One may expect a similar result to hold when multiplying on the left by $p$ a relation of the form $q_0\doms q_1$, and indeed it was claimed (without proof) in~\cite{hhm} that $\equidom$ is a congruence with respect to $\otimes$. Unfortunately, this turns out not be true in general: we will see in Section~\ref{section:cntrex} that it is possible to have  $q_0\equidom q_1$ and $p\otimes q_0\ndoms p\otimes q_1$ simultaneously. For the time being, we assume this does not happen as an hypothesis and explore some of  its immediate consequences.
\begin{defin}
For a theory $T$, we say that \emph{$\otimes$ respects} (or \emph{is compatible with}) $\doms$ (resp.~$\equidom$)  iff for all global invariant types $p$, $q_0$, $q_1$, if $q_0\doms q_1$ (resp.~$q_0\equidom q_1$) then  $p\otimes q_0\doms p\otimes q_1$ (resp.~$p\otimes q_0\equidom p\otimes q_1$).
\end{defin}

\begin{co}\label{co:criticalpart}\*
  \begin{enumerate}
  \item $\otimes$ respects $\doms$ if and only if
    $(\invtypes(\monster), \otimes, \doms)$ is a preordered
    semigroup. In this case $\domeq$ is a congruence with respect to $\otimes$, and the latter induces on $(\invtilde, \doms)$ the structure of a
    partially ordered semigroup.
  \item    $\otimes$ respects $\equidom$ if and only if  $\equidom$ is a congruence with respect to $\otimes$.
\end{enumerate}
\end{co}

\begin{proof}
Everything follows at once from Lemma~\ref{lemma:atleastontheleft}.
\end{proof}

\begin{lemma}
  Suppose that $p,q\in \invtypes(\monster)$ and $p$ is realised. The following are equivalent:
  \begin{enumerate*}
  \item $p\equidom q$.
  \item $p\domeq q$.
  \item $p\doms q$.
  \item $q$ is realised.
  \end{enumerate*}
\end{lemma}
\begin{proof}
The implications $1\allora 2\allora 3$ are true by definition, even when $p$ is not realised.  Let $p=\tp(a/\monster)$, where $a\in \monster$.

  For $3\allora 4$ suppose that $r\in S_{pq}(A)$ is such that   $p\cup r\proves q$. Since $\set{x=a}\proves p$, we have $\set{x=a}\cup r\proves q$. But since $\set{x=a}\cup r$ is a small  type, it is realised in $\monster$ by some $(a,b)$, and clearly $b\models q$.

For $4\allora 1$ suppose that for some $b\in \monster$ we have  $q=\tp(b/\monster)$ and let $A$  be  any small set containing $a$ and $b$. Clearly, $(x=a)\land (y=b)$ implies a complete type $r\in S_{xy}(A)$ containing $(p\restr A)\cup (q\restr A)$, and since $r(x,y)\proves p(x)\cup q(y)$ we have that $r$ witnesses $p\equidom q$.
\end{proof}
\begin{lemma}\label{lemma:ignorerealised}
  Suppose that $p_x, q_y\in \invtypes(\monster)$ and that $p$ is realised by $a\in \monster$. Then $\set{x=a}\cup q(y)\proves p(x)\cup q(y)\proves p(x)\otimes q(y)=q(y)\otimes p(x)$. Moreover, $p(x)\otimes q(y)\equidom q(y)$.
\end{lemma}
\begin{proof}
  The first part is clear. It follows that, if $q$ is $A$-invariant and $a\in A$, in order to show that $p(x)\otimes q(y)\equidom q(z)$ it suffices to take as $r$ the type $\set{x=a}\cup \set{y=z}\cup (q(y)\restr A)\cup (q(z)\restr A)$.
\end{proof}

\begin{notation}
When quotienting by  $\domeq$ or $\equidom$ we denote by $\class{p}$ the class of $p$, with the understanding that the equivalence relation we are referring to is clear from context.  We write  $\class{0}$  for the class of realised types.\footnote{It is the class of the unique global $0$-type.} 
\end{notation}

\begin{pr}\label{pr:realneut}
  Suppose that $\otimes$ respects $\doms$ (resp.~$\equidom$). Then:
  \begin{enumerate}
  \item $(\invtilde, \otimes)$ (resp.~$(\invbar, \otimes)$) has neutral element $\class 0$;
  \item no element different from $\class{0}$ is invertible;
  \item \label{point:min} in $\invtilde$, $\class 0$ is the minimum of $\doms$.
  \end{enumerate}
\end{pr}
\begin{proof}\*
  \begin{enumerate}
  \item Let $p=\tp(a/\monster)$ and $q\in \invtypes(\monster)$, where $a\in \monster$. Apply Lemma~\ref{lemma:ignorerealised} and note that $p\otimes q\equidom q$ implies $p\otimes q\domeq q$.
  \item By the previous point if $\class p$ is invertible then there is some $q\in\invtypes(\monster)$ such that $p\otimes q$ is realised. In particular, $p$ is realised as well.
   \item We have to show that for every $p(x)$ and every realised $q(y)$ we have $p\doms q$. If $q$ is realised by $b\in \monster$, it is sufficient to put in $r$ the formula $y=b$.\qedhere 
  \end{enumerate}
\end{proof}

\subsection{Some Sufficient Conditions}
We proceed to investigate sufficient  conditions for $\otimes$ to respect $\doms$ and $\equidom$. These conditions are admittedly rather artificial, but we show they are a consequence of other properties that are easier to test directly, such as stability.

In what follows, types will be usually assumed to have no realised coordinates and no duplicate coordinates, i.e.\ we will assume, for all $i<j<\abs x$ and $a\in \monster$, to have $p(x)\proves (x_i\ne a)\land  (x_i\ne x_j)$. Up to domination-equivalence, and even equidominance, no generality is lost, as justified by Lemma~\ref{lemma:ignorerealised} and by the fact that, for example, if $p(x_0)$ is any $1$-type and  $q(y_0,y_1)\proves p(y_0)\cup \set{y_0=y_1}$, setting $x_0=y_0=y_1$ shows $p\equidom q$.

 We usually abuse the notation and indicate e.g.\ $(\invtilde, \otimes, \domd)$ simply with $\invtilde$.

 \begin{defin}\label{defin:statdom}
Let $p,q_0,q_1\in\invtypes(\monster)$ and $r\in S_{q_0q_1}(A)$. Let $\monster_1\satext \monster$ and $b,c\in \monster_1$ be such that $(b,c)\models q_0\cup r\cup q_1$, and let $a\models p(x)\invext \monster_1$. Define \[r[p]\ceq (\tp_{xyz}(abc/A)\cup \set{x=w})\in S_{xyzw}(A)\]
   
 We say that $T$ has \emph{stationary domination} (resp.~\emph{stationary equidominance}) iff  whenever $p,q_0,q_1\in\invtypes(\monster)$ and $q_0\doms q_1$ (resp.~$q_0\equidom q_1$), there are   $A\smallsubset \monster$  and $r\in S_{q_0q_1}(A)$ such that
 \begin{itemize}
 \item $p(x),q_0(y),q_1(z)$ are $A$-invariant,
 \item $q_0\cup r\proves q_1$ (resp.~$q_0\cup r\proves q_1$ and $q_1\cup r\proves q_0$), and
 \item for all $\monster_1\satext \monster$, all $b,c\in \monster_1$ such that $(b,c)\models q_0\cup r$ and all $a\models p(x)\invext \monster_1$, we have \[(p(x)\otimes q_0(y))\cup r[p]\proves p(w)\otimes q_1(z)\] (resp.~$(p(x)\otimes q_0(y))\cup r[p]\proves p(w)\otimes q_1(z)$ and $(p(w)\otimes q_1(z))\cup r[p]\proves p(x)\otimes q_0(y)$).
\end{itemize}
\end{defin} 

\begin{pr}\label{pr:obvwks}
  If $T$ has stationary domination, then $\otimes$ respects $\doms$.   If $T$ has stationary equidominance, then $\otimes$ respects $\equidom$.
\end{pr}
\begin{proof}
  Immediate from the definitions.
\end{proof}

\begin{defin}
We say that \emph{$q_1$ is algebraic over $q_0$} iff  there are $b\models q_0$ and $c\models q_1$ such that $c\in\acl(\monster b)$.  We say that $T$ has \emph{algebraic domination} iff  $p\doms q$ if and only if $q$ is algebraic over $p$.
\end{defin}
\begin{pr}\label{pr:algdom}
  Suppose that $q_1$ is algebraic over $q_0$.  Then for all $p\in \invtypes(\monster)$ we have $p\otimes q_0\doms p\otimes q_1$, and this is witnessed by a type $r[p]$ as in the definition of stationary domination. In particular, algebraic domination implies stationary domination.
\end{pr}
\begin{proof}
  Let $b,c\in\monster_1\satext \monster$ witness algebraicity of $q_1$ over $q_0$. Suppose $\psi(y,z)$ is an  $L(\monster)$-formula such that $\psi(b,z)$ isolates $\tp(c/\monster b)$, and let $s\coloneqq\set{x=w}\cup \set{\psi(y,z)}$. If $A$ is such that $p\in \invtypes(\monster, A)$ and $\psi(y,z)\in L(A)$, let $r\coloneqq q_0(y)\restr A\cup \set{\psi(y,z)}$. Then  $s\subseteq r[p]$, so it is enough to show that  $p(x)\otimes q_0(y)\cup s\proves p(w)\otimes q_1(z)$. In some $\monster_2\satext \monster_1$, let $a\models p\invext \monster_1$ and let $\phi(w, z)\in p(w)\otimes q_1(z)$. This means that $\phi(w,z)\in L(\monster)$ and $\phi(w, c)\in \tp(a/\monster c)=p\invext \monster c$.

  By hypothesis, there are only finitely many  $\tilde c\equiv_{\monster b} c$, which must be contained in any model containing $\monster b$ and, by invariance of $p\invext \monster_1$, for all such $\tilde c\in \monster_1$ we have  $p\invext \monster_1\proves \phi(x, \tilde c)$.    It follows that $\tp(a/\monster_2)\proves \forall z\; \bigl(\psi(b,z)\implica \phi(w,z)\bigr)$.
As the latter is an $L(\monster b)$-formula, it is contained in $p\invext \monster b$, and it follows that $(p(x)\otimes q_0(y))\cup s\proves p(w)\otimes q_1(z)$. 
\end{proof}
\begin{co}\label{co:pushforward}
  Suppose that $q_1$ is the pushforward $f_*(q_0)$ of $q_0$, for some definable function $f$. Then, for all $p\in \invtypes(\monster)$, we have $p\otimes q_0\doms p\otimes q_1$.
\end{co}

  \begin{pr}\label{pr:stabstdom}
    Let $T$ be stable. Then $T$ has stationary domination and stationary equidominance. Moreover $\invtilde$ and $\invbar$ are commutative.
  \end{pr}
  \begin{proof}
Let $r$  witness $q_0(y)\doms q_1(z)$.    By Lemma~\ref{lemma:atleastontheleft} $(q_0(y)\otimes p(x))\cup r(y,z)\cup\set{x=w}\proves q_1(z)\otimes p(w)$. As $r\cup \set{x=w}\subseteq r[p]$ and $T$ is stable if and only if $\otimes$ is commutative, we have stationary domination and commutativity of $\invtilde$. For stationary equidominance and commutativity of $\invbar$, argue analogously starting with any $r$ witnessing $q_0(y)\equidom q_1(z)$.
  \end{proof}
  \begin{defin}
$T$ is \emph{weakly binary} iff whenever $a,b$ are tuples from some $\monster_1\succ \monster$ and $\tp(a/\monster)$ and $\tp(b/\monster)$ are invariant there is $A\smallsubset \monster$ such that
    \begin{equation}\label{eq:weakbin}
    \tp(a/\monster)\cup \tp(b/\monster)\cup \tp(a,b/A)\proves \tp(a,b/\monster)
  \end{equation}
\end{defin} 
\begin{lemma}\label{lemma:stillinvariant}
 If $T$ is weakly binary and $\tp(a/\monster)$, $\tp(b/\monster)$ are both invariant, then so is $\tp(ab/\monster)$.
\end{lemma}
\begin{proof}
If~\eqref{eq:weakbin} holds and  $\tp(a/\monster)$ and $\tp(b/\monster)$ are $B$-invariant then the left-hand side of~\eqref{eq:weakbin} is fixed by $\aut(\monster/AB)$. As  $\tp(a,b/\monster)$ is complete, it is $AB$-invariant.
\end{proof}
\begin{eg}
  Every binary theory $T$, i.e.\ where every formula is equivalent modulo $T$ to a Boolean combination of formulas with at most two free variables, is weakly binary. This follows from the fact that $T$ is binary if and only if for any $B$ and tuples $a,b$
  \[
    \tp(a/B)\cup \tp(b/B)\cup \tp(ab/\emptyset)\proves \tp(ab/B)
  \]
  An example of a weakly binary theory which is not binary is the theory of a dense circular order, or any other non-binary theory that becomes binary after naming some constants.   A weakly binary theory which does not become binary after adding constants can be obtained by considering a structure $(M, E, R)$ where $E$ is an equivalence relation with infinitely many classes, on each class $R(x,y,z)$ is a circular order, and $R(x,y,z)\implica E(x,y)\land E(x,z)$.
 The generic $3$-hypergraph and $\textsf{ACF}_0$ are not weakly binary.
\end{eg}

We thank Jan Dobrowolski for pointing out the relationship between binarity and weak binarity,  therefore also implicitly suggesting a name for the latter.
\begin{lemma}
$T$ is weakly binary if and only if for every $n\ge 2$ we have the  following.  If $a^0,\ldots,a^{n-1}$ are such that for all $i<n$ we have $\tp(a^i/\monster)\in \invtypes(\monster)$, then there is $A\smallsubset \monster$ such that 

\begin{equation}
\Bigl(\bigcup_{i=0}^{n-1}    \tp(a^i/\monster)\Bigr)\cup \tp(a^0,\ldots, a^{n-1}/A)\proves \tp(a^0,\ldots,a^{n-1}/\monster)\label{eq:wbinn}
\end{equation}
\end{lemma}
\begin{proof}
For the nontrivial direction, assume $T$ is weakly binary. For notational simplicity we will only show the case $n=3$, and leave the easy induction to the reader.  Let  $a,b,c$  be tuples with invariant global type. By Lemma~\ref{lemma:stillinvariant} $\tp(bc/\monster)$ is still invariant, so we can let $A$ witness weak binarity for $b,c$ and for $a,bc$ simultaneously, where $bc$ is considered now as a single tuple. Then  $\tp(b/\monster)\cup \tp(c/\monster)\cup \tp(a,b,c/A)\proves \tp(b,c/\monster)$, and by applying weak binarity to $a, bc$ we get
  \begin{align*}
    &\phantom{\proves}\tp(a/\monster)\cup \tp(b/\monster)\cup \tp(c/\monster)\cup \tp(a,b,c/A)\\
    &\proves \tp(a/\monster)\cup \tp(bc/\monster)\cup \tp(a,bc/A)\\
    &\proves \tp(a,bc/\monster)\qedhere
  \end{align*}
\end{proof}

\begin{co}\label{co:binstat}
  Every weakly binary theory has stationary domination and stationary equidominance.
\end{co}
\begin{proof}
  Let $p(x),q_0(y), q_1(z)$ be $A_0$-invariant and $r\in S_{q_0q_1}(A_0)$ be such that $q_0\cup r\proves q_1$.
  In some $\monster_1\satext \monster$ choose $(b,c)\models q_0 \cup r$, then choose $a\models p\invext \monster_1$. By the case $n=3$ of~\eqref{eq:wbinn} there is some $A\smallsubset \monster$, which without loss of generality includes $A_0$, such that
  \begin{equation}
    \tp(a/\monster)\cup \tp(b/\monster)\cup \tp(c/\monster)\cup\tp(abc/A)\proves \tp(abc/\monster)\label{eq:wbimsd}
  \end{equation}
  Let $r[p]\ceq \tp_{xyz}(abc/A)\cup \set{x=w}$ and note that $r\subseteq r[p]$. Therefore  $(p\otimes q_0)\cup r[p]\proves q_0\cup r\proves q_1=\tp(c/\monster)$. Combining this with~\eqref{eq:wbimsd}, and observing that $\tp(ab/\monster)=p\otimes q_0$, that $\tp(ac/\monster)=p\otimes q_1$ and that $r[p]\proves x=w$, we have 
  \begin{align*}
    &\phantom{\proves}\bigl(p(x)\otimes q_0(y)\bigr)\cup r[p]\\
    &\proves \bigl(p(x)\otimes q_0(y)\bigr)\cup r[p]\cup q_1(z)\cup \set{x=w}\\
    &\proves  \tp_x(a/\monster)\cup \tp_y(b/\monster)\cup \tp_z(c/\monster) \cup\tp_{xyz}(abc/A)\cup\set{x=w}\\
    &\proves \tp_{wz}(ac/\monster)=p(w)\otimes q_1(z)
  \end{align*}
  This proves stationary domination. For stationary equidominance, start with an $r$ witnessing $q_0\equidom q_1$ and prove analogously that in addition $(p(w)\otimes q_1(z))\cup r[p]\proves p(x)\otimes q_0(y)$.
\end{proof}

  We now give some examples of $(\invtilde, \otimes, \doms)$. These characterisations can be proven with easy  ad hoc arguments but, as such computations are made almost immediate by results like Proposition~\ref{pr:wortpreserved} or Theorem~\ref{thm:bigoplusN}, we state them without proof. We postpone the investigation of further examples to a future paper.
  \begin{eg}\label{eg:strminN}
    If $T$ is a strongly minimal theory (see Example~\ref{eg:deq}), then $(\invtilde, \otimes, \doms)\cong (\mathbb N, +, \ge)$.
  \end{eg}

\begin{eg}
  Let $T$ be the theory of an equivalence relation $E$ with infinitely many classes, all of which are infinite. Since $T$ is $\omega$-stable, by Proposition~\ref{pr:stabstdom} and Proposition~\ref{pr:obvwks} $\otimes$ respects $\doms$, and moreover  by~\cite[Theorem~14.2]{poizat} for every $\kappa$ there is a $\kappa$-saturated $\monster\models T$ of size $\kappa$. For such $\monster$ we have  $(\invtilde, \otimes, \doms)\cong \bigoplus_{\kappa} \mathbb N$, where each copy of $\mathbb N$ is equipped  with the usual $+$ and $\ge$,  and $\oplus$ is the direct sum of ordered monoids.

  To spell this out and give a little extra information on $\invtilde$ for $T$, fix a choice of representatives $\seq{b_i\mid 0<i<\kappa}$ for $\monster/E$ and let  $\pi_E\from \monster\to \monster/E$ be the  projection to the quotient. Then an element $\class p\in\invtilde$ corresponds to a $\kappa$-sequence $(n_i)_{i<\kappa}$ of natural numbers with finite support where,   for any $c\models p$,   $n_0=\abs{\pi_E c\setminus \pi_E \monster}$ and, for positive $i$, $n_i=\abs{\set{c_j\in c\mid E(c_j, b_i)}\setminus E(\monster, b_i)}$, i.e.\  $n_0$ counts the new equivalence classes represented in $p$ and, when $i$ is positive, $n_i$ counts the number of new points in the equivalence class of $b_i$. Addition is done componentwise and $(n_i)_{i<\kappa}\le (m_i)_{i<\kappa}$ iff $\forall i<\kappa\; n_i\le m_i$. 
\end{eg}
  As we will see in Section~\ref{section:stabletheories}, the fact that $\invtilde$ has the previous forms follows from the stability-theoretic properties of the theories above: Theorem~\ref{thm:bigoplusN} applies to  both and, in the case of Example~\ref{eg:strminN},   Corollary~\ref{co:unidim} tells us directly that $\invtilde\cong \mathbb N$.
\begin{eg}\label{eg:dlo}
As \textsf{DLO} is binary, $\otimes$ respects $\doms$. We have already seen an example of two domination-equivalent types in this theory in  Example~\ref{eg:deq}. To describe $\invtilde$, call a cut in   $\monster$ \emph{invariant} iff it has small cofinality on exactly one side, and let $\operatorname{IC}\monster$ be the set of all such. The domination-equivalence class of an invariant type in \textsf{DLO} is determined by the (necessarily invariant) cuts in which it concentrates  and, writing $\pfin(X)$ for the set of finite subsets of $X$, we have  \mbox{$(\invtilde, \otimes, \doms)\cong (\pfin(\operatorname{IC}\monster), \cup, \supseteq)$}. 
\end{eg}

\section{Counterexamples}\label{section:cntrex}
In~\cite[p.~18]{hhm} it was claimed without proof that  $\invbar$ is well-defined and commutative in every first-order theory. This section contains counterexamples to the statements above.
        \subsection{Well-Definedness}\label{subsec:wdbreaks}
This subsection is dedicated to the proof of the following result.
\begin{thm}\label{thm:cntrexwdsu2}
  There is a ternary, $\omega$-categorical, supersimple theory of SU-rank $2$ with degenerate algebraic closure in which neither $\domeq$ nor $\equidom$ are congruences with respect to $\otimes$.
\end{thm}
In Proposition~\ref{pr:descriptionofT}, we present the promised theory  as a Fra\"iss\'e limit (see \cite[Theorem~7.1.2]{hodges}) and provide an explicit axiomatisation. We then show in Proposition~\ref{pr:dombreaks} that in this theory $\otimes$ does not respect $\doms$, nor $\equidom$.

Denote by $S_3$ the group of permutations of $\set{0,1,2}$. 
\begin{defin}
  Let $L$ be the relational language $L\coloneqq\set{E^{(2)}, R_2^{(2)},  R_3^{(3)}}$, where  arities of symbols are indicated as superscripts, and define $\Lambda\coloneqq \Lambda_0\land \Lambda_1$, where 
  \begin{gather*}
    \Lambda_0(x_0, x_1, x_2)\coloneqq\bigvee_{\sigma\in S_3}\bigl(R_2( x_{\sigma0}, x_{\sigma1})\land R_2( x_{\sigma0}, x_{\sigma2})\land \neg R_2( x_{\sigma1}, x_{\sigma2})\bigr)\\
\Lambda_1(x_0, x_1, x_2)\coloneqq    \bigwedge_{0\le i<j<3} \neg E(x_i, x_j)
  \end{gather*}
Let $K$ be the class of finite $L$-structures where
\begin{enumerate}
\item $E$ is an equivalence relation, 
\item $R_2$ is symmetric, irreflexive
  and $E$-equivariant, i.e.\ $(E(x_0, x_1)\land E(y_0,y_1))\implica (R_2(x_0, y_0)\coimplica
  R_2(x_1, y_1))$,
\item $R_3$ is a symmetric relation, i.e.\ $R_3(x_0,x_1, x_2)\to\bigwedge_{\sigma\in S_3} R_3( x_{\sigma0},  x_{\sigma1}, x_{\sigma2})$, and
\item   $R_3(x_0,x_1,x_2)\to \Lambda(x_0,x_1, x_2)$ is satisfied.
\end{enumerate}

\end{defin}
Note that in particular $R_2$ is still symmetric irreflexive on the quotient by $E$.  We do not add an imaginary sort for this quotient; it will be notationally convenient to mention it anyway but, formally, every reference to the quotient by $E$, the relative projection, etc, is to be understood as a mere shorthand.
\begin{pr}\label{pr:descriptionofT}
  \begin{enumerate}
  \item $K$ is a Fra\"iss\'e class with strong amalgamation.
  \end{enumerate}
  \smallskip
Let $T$ be the theory of the Fra\"iss\'e limit of $K$.
\begin{enumerate}[resume]
\item  $T$ is $\omega$-categorical, eliminates quantifiers in $L$ and has degenerate algebraic closure, i.e.\ for all sets $X\subseteq M\models T$ we have $\acl X=X$. 
\item $T$ is ternary, i.e.\ in $T$ every formula  is equivalent to a Boolean combination of formulas with at most $3$ free variables.
\item $T$ can be axiomatised as follows:
  \begin{enumerate}[label=(\Roman*)]
  \item\label{point:eqrel} $E$ is an equivalence relation with infinitely many classes, all of which are infinite.
  \item Whether $R_2(x_0,x_1)$ holds only depends on the $E$-classes of $x_0$, $x_1$; moreover, the structure induced by $R_2$  on the quotient by $E$ is elementarily equivalent to the Random Graph.
  \item \label{point:crucialaxiom}$T$ satisfies   $R_3(x_0,x_1,x_2)\to \Lambda(x_0,x_1, x_2)$, i.e.\ if $R_3(x_0,x_1,x_2)$ holds then between the $x_i$  there are precisely two $R_2$-edges and their $E$-classes are pairwise distinct.
  \item \label{point:gthg}Denote by $[x_i]_E$ the $E$-class of $x_i$. If $\Lambda(x_0, x_1, x_2)$ holds, then $R_3\restr [x_0]_E\times [x_1]_E\times [x_2]_E$ is a symmetric generic tripartite $3$-hypergraph, i.e.\ for any $i<j<3$ and $k\in \set{0,1,2}\setminus \set{i,j}$, if $U, V\subseteq [x_i]_E\times [x_j]_E$ and $U\cap V=\emptyset$ then there is $z\in [x_k]_E$ such that for every $(x,y)\in U$ we have $R_3(x,y,z)$ and for every $(x,y)\in V$ we have $\neg R_3(x,y,z)$.
  \end{enumerate}
\item $T$ is supersimple of SU-rank $2$.
\end{enumerate}
\end{pr}
\begin{proof}\*
  \begin{enumerate}
  \item Routine, left to the reader.
  \item This is standard, see e.g.~\cite[Theorem~7.1.8 and Corollary~7.3.4]{hodges}.
  \item $T$ eliminates quantifiers in a ternary relational language. 
  \item Easy back-and-forth between the Fra\"iss\'e limit of $K$ and any model of \ref{point:eqrel}--\ref{point:gthg}.
  \item Denote by $\pi$  the projection to the quotient by $E$. A routine application of the Kim-Pillay Theorem (see~\cite[Theorem~4.2]{kimpillay}) shows that $T$ is simple and forking is given by    $a\find C b\iff (a\cap b\subseteq C) \land (\pi a\cap \pi b\subseteq \pi C)$, from which we immediately see that the SU-rank of any $1$-type in $T$ is at most $2$; finding a $1$-type of SU-rank $2$ is easy.\qedhere
  \end{enumerate}
\end{proof}

\begin{defin}\label{defin:tfc}
 In $T$, define the global types
 \begin{align*}
& p(x)\coloneqq\set{R_2(x,a)\mid a\in \monster}\cup \set{\neg R_3(x,a,b)\mid a,b\in \monster}\\
 &q_0(y)\coloneqq\set{\neg R_2(y,a)\mid a\in \monster}\\
 &q_1(z_0, z_1)\coloneqq\set{\neg R_2(z_0,a)\mid a\in \monster}\cup\set{E(z_0, z_1)\land z_0\ne z_1}
\end{align*}
\end{defin}

These three types are complete by quantifier elimination and the axioms of $T$: for instance, in the case of $q_1$, the condition $E(z_0, z_1)$ together with the restriction of $q_1$ to $z_0$ decides all the $R_2$-edges of $z_1$, and for all $a,b\in \monster$ we have  $\neg\Lambda_0(z_1, a,b)$, hence $\neg R_3(z_1, a,b)$. Moreover, it follows easily from their definition that $p$, $q_0$ and $q_1$ are all $\emptyset$-invariant.

\begin{pr}\label{pr:dombreaks}
  $q_0\equidom q_1$ and in particular $q_0\domeq q_1$.  Nonetheless, $p(x)\otimes q_0(y)\ndoms p(w)\otimes q_1(z)$.
\end{pr}
\begin{proof}
Let $A$ be any small set and let $r(y,z_0, z_1)\in S_{q_0q_1}(A)$ contain the formula $y=z_0$. Clearly, $q_1(z)\cup r(y,z)\proves q_0(y)$. Moreover, since $E(z_0, z_1)\land z_0\ne z_1\in (q_1\restr\emptyset)\subseteq r$ we have the first part of the conclusion.

Note that $p(x)\otimes q_0(y)$  is axiomatised by \[p(x)\cup q_0(y)\cup \set{R_2(x, y)}\cup\set{ \neg R_3(x, y, a)\mid  a\in \monster}\] and similarly $p(w)\otimes q_1(z)$ is axiomatised by \[p(w)\cup q_1(z)\cup \set{R_2(w, z_0)\land R_2(w, z_1)}\cup\set{ \neg R_3(w, z_j, a)\mid j< 2, a\in \monster}\]  Let $A$ be any small set and $r(x,y,w,z)\in S_{p\otimes q_0, p\otimes q_1}(A)$, then pick any  $a\in \monster\setminus A$ and  $i<2$  such that  $(p(x)\otimes q_0(y))\cup r\proves y\ne z_i$. By genericity of $R_2$, the set    \[\Phi\coloneqq (p(x)\otimes q_0(y))\cup r\cup\set{R_2(w,a)\land R_2(w,z_i)\land \neg R_2(z_i, a)}\] is consistent\footnote{E.g.\ if $r\proves x=w\land y=z_0$ then we even have $p(x)\otimes q_0(y)\cup r\proves R_2(w,a)\land R_2(w,z_1)\land \neg R_2(z_1, a)$.
}, and by genericity of $R_3$ so is $\Phi\cup \set{R_3(w, z_i, a)}$  (as well as $\Phi\cup \set{\neg R_3(w, z_i, a)}$). This shows that \[p(x)\otimes q_0(y)\cup r\centernot \proves \set{\neg R_3(w, z_j, a)\mid j<2, a\in \monster}\subseteq p(w)\otimes q_1(z)\qedhere\]
\end{proof}
As an aside, note that anyway $p(x)\otimes q_0(y) \domd p(w)\otimes q_1(z)$ by Corollary~\ref{co:pushforward}, the map $f$ being the projection on the coordinates $(w,z_0)$.
\begin{rem}
  An inspection of the proof shows that in fact $q_0$ and $q_1$ are ``equidominance-semi-isolated'' over each other (\emph{strongly RK-equivalent} in the terminology of~\cite{tanovic}), i.e.\   there is a formula $\phi(y,z)$ consistent with $q_0(y)\cup q_1(z)$ such that $q_0(y)\cup \set{\phi(y,z)} \proves q_1(z)$ and $q_1(z)\cup \set{\phi(y,z)} \proves q_0(y)$; in this case we can take $\phi\coloneqq y=z_0\land E(z_0, z_1)\land z_0\ne z_1$. Therefore the same counterexample also works with this finer equivalence relation.
\end{rem}

\begin{question}
  Is $\invtilde$ well-defined in every \textsf{NIP} theory?
\end{question}

\subsection{Commutativity}
In this subsection we prove that  in the theory of the Random Graph $\invtilde$ coincides with $\invbar$ and is not commutative. To begin with, note that this theory is binary, hence $\invtilde$ is well-defined by Corollary~\ref{co:binstat} and Proposition~\ref{pr:obvwks}. This also follows from the characterisation of domination we are about to give in Proposition~\ref{pr:rgdegdom}.
\begin{defin}\label{defin:degdom}
Let $L_0$ be the ``empty'' language, containing only equality. We say that $T$ has \emph{degenerate domination}  iff whenever $p(x)\doms q(y)$  there is a small set $r_0$ of $L_0(\monster)$ formulas with free variables included in $xy$ and consistent with $p$ such that $p\cup r_0\proves q$.
\end{defin}
\begin{rem}\label{rem:dgdi}
It is easy to see that, if there is $r_0$ as above, then $q$ is included in $p$ up to removing realised and duplicate coordinates  and renaming  the remaining ones. 
\end{rem}

\begin{lemma}\label{lemma:degdom}
Suppose $T$ has degenerate domination. Then $T$ has algebraic domination, and in particular $\otimes$ respects $\doms$. Moreover for global types $p$ and $q$ the following are equivalent:
  \begin{enumerate}
  \item \label{point:trivdeq} There is a small set $r_0$ of $L_0(\monster)$ formulas consistent with $p\cup q$ such that $p\cup r_0\proves q$ and $q\cup r_0\proves p$.
  \item $p\equidom q$.
  \item \label{point:deq} $p\domeq q$.
  \end{enumerate}
  In particular, $\otimes$ respects $\equidom$ too.
\end{lemma}
\begin{proof}
By Remark~\ref{rem:dgdi} degenerate domination implies algebraic domination. The implications $1\allora 2\allora 3$ are trivial and hold in any theory.  To prove $3\allora 1$ suppose $p(x)\domeq q(y)$, and let $r_1$ and $r_2$ be small sets of $L_0(\monster)$ formulas with free variables included in $xy$ and consistent with $p\cup q$ such that $p\cup r_1\proves q$ and $q\cup r_2\proves p$. It follows easily from Remark~\ref{rem:dgdi} that we may find $r_0$ satisfying the same restrictions as $r_1$ and $r_2$ and such that $p\cup r_0\proves q$ and $q\cup r_0\proves p$ hold simultaneously. 
\end{proof}
\begin{pr}\label{pr:rgdegdom}
  The Random Graph has degenerate domination.
\end{pr}
\begin{proof}
  Suppose that $r\in S_{pq}(A)$ witnesses $p(x)\doms q(y)$ and assume that $q$ has no realised or duplicate coordinates. Up to a permutation of the $y_j$, assume that $r$ identifies  $y_0,\ldots, y_{n-1}$ with some variables in $x$ and for all $j$ such that $n\le j<\abs y$ and all $i<\abs x$ we have $r\proves x_i\ne y_j$.  If $n=\abs y$ then we can let $r_0$ be a suitable restriction of $r$ and we are done, so assume that $n<\abs y$, hence for every $i<\abs x$ we have  $r\proves y_n\ne x_i$.  Pick any $b\in \monster\setminus A$;   by the Random Graph axioms  $p\cup  r$ is consistent with both $E(y_n, b)$ and $\neg E(y_n, b)$, contradicting $p\cup  r \proves q$.
\end{proof}
\begin{co}\label{co:rgncomm}
  In the theory of the Random Graph, $\invtilde(=\invbar)$ is not commutative.
\end{co}
\begin{proof}
  Consider the global types  $p(x)\coloneqq \set{\neg E(x, a)\mid a\in \monster}$ and $q(y)\coloneqq \set{E(y, a)\mid a\in \monster}$. Both are clearly $\emptyset$-invariant, and it follows straight from the definitions that $p(x)\otimes q(y)\proves \neg E(x,y)$ and $q(y)\otimes p(x)\proves E(x,y)$. The conclusion now follows from degenerate domination and Remark~\ref{rem:dgdi}.
\end{proof}

 Other easy consequences of Proposition~\ref{pr:rgdegdom} are that in the theory of the Random graph
 \begin{enumerate}
 \item $\invtilde$  is not generated by the classes of the $n$-types for any fixed $n<\omega$, 
 \item $\invtilde$ is not generated by any family of classes of pairwise weakly orthogonal types (see Definition~\ref{defin:wort}), and 
 \item for any nonrealised $p$ the submonoid generated by  $\class p$  is infinite.
 \end{enumerate}

\begin{question}
Let $T$ be \textsf{NIP} and assume $\invtilde$ is well-defined. Is it necessarily commutative?
\end{question}

The analogous question for $\invbar$ has a negative answer. We are grateful to E.~Hrushovski for pointing out the following counterexample and allowing us to include it.

Let \textsf{DLOP} be as in Example~\ref{eg:deq}. It eliminates quantifiers in $\set{<, P}$, it is \textsf{NIP}, and it is binary, hence $\invtilde$ and  $\invbar$ are well-defined by Corollary~\ref{co:binstat} and Proposition~\ref{pr:obvwks}.

\begin{pr}[Hrushovski]\label{pr:dlopncomm}
In $\textsf{DLOP}$, $\invbar$ is not commutative.
\end{pr}
\begin{proof}
Let $p$ be the type at $+\infty$ in the predicate $P$ and $q$ the type at $+\infty$ in $\neg P$, and note that both types are $\emptyset$-invariant. Let $r\in S_{p\otimes q, q}(\emptyset)$  contain the formula $y=z$. Then $r$ witnesses $p_x\otimes q_y\equidom q_z$, and similarly one shows that $q\otimes p\equidom p$. As shown in Example~\ref{eg:deq}, $p$ and $q$ are not equidominant, and therefore we have  $(p\otimes q)\equidom q\nequidom p\equidom (q\otimes p)$.
\end{proof}
This counterexample exploits crucially $\equidom$, as opposed to $\domeq$. In fact, in \textsf{DLOP} $\invtilde$ is the same as in the restriction of $\monster$ to $\set <$, and in \textsf{DLO} $\invtilde$ is commutative. A further analysis 
also shows that $(\invbar, \otimes)$ cannot be endowed with any order $\le$ compatible  with $\otimes$ in which $\class 0$ is the minimum. In fact, if $p$ and $q$ are as above, then we have already shown that  $(p\otimes q)\equidom q\nequidom p\equidom (q\otimes p)$. If we had an order $\le$ as above then we would get
\[
  \class{p}=\class{p}\otimes\class 0\le\class p\otimes \class q=\class q=\class{q}\otimes\class 0\le\class q\otimes \class p=\class p
\]
contradicting $\class p \ne \class q$.

\section{Properties Preserved by Domination}\label{section:prpr}
In this section we show that some properties  are preserved downwards by domination. These invariants also facilitate  computations of $\invtilde$ and $\invbar$ for specific theories; an immediate consequence is for instance Corollary~\ref{co:invbarteqchange}, that such monoids may change when passing to $T^\eq$.

The next results are related to the ones in~\cite{tanovic}, which contains a study of weak orthogonality and the global RK-order (similar to domination) in the case of generically stable regular types. Of particular interest are~\cite[Proposition~3.6]{tanovic}, to which Theorem~\ref{thm:propertiespreserved} is related, and~\cite[Theorem~4.4]{tanovic}.
\subsection{Finite Satisfiability, Definability, Generic Stability}
\begin{defin}
  Let $p\in \invtypes_x(\monster, A)$. A \emph{Morley sequence} of $p$ over $A$ is an $A$-indiscernible sequence $\seq{a^i\mid i\in I}$, indexed on some totally ordered set $I$,   such that for any $\bla i0<{n-1}$ in $I$ we have $\tp(a^{i_{n-1}},\ldots, a^{i_0}/A)=\pow pn\restr A$ [sic]\footnote{E.g.\ $(a^1, a^0)\models (p(x^1)\otimes p(x^0))\restr A$. This awkwardness in notation is an unfortunate consequence of the order in which $\otimes$ is written, i.e.\ realising the type on the right first.}.  
\end{defin}

\begin{defin}
Let $M\smallprec \monster$ and $A\smallsubset \monster$.  
\begin{enumerate}
\item A partial type $\pi$ is \emph{finitely satisfiable in $M$} iff for every finite conjunction $\phi(x)$ of formulas in $\pi$ there is $m\in M$ such that $\models \phi(m)$.
\item A global type $p\in S_x(\monster)$ is \emph{definable over $A$} iff it is $A$-invariant and for every $\psi(x;y)\in L$ the set  $d_p\psi$ is clopen, i.e.\ of the form $\set{q\in S_y(A)\mid \phi\in q}$ for a suitable $\phi\in L(A)$.
\item A global type $p\in S_x(\monster)$ is \emph{generically stable over $A$} iff it is $A$-invariant and for every ordinal $\alpha\ge\omega$ and Morley sequence $(a^i\mid i<\alpha)$ of $p$ over $A$, the set of formulas $\phi(x)\in L(\monster)$ true of all but finitely many $a^i$ is a complete global type. 
\end{enumerate}
We say that $p$ is \emph{definable} iff it is definable over $A$ for some small $A$, and similarly for the other two notions. 
\end{defin}
The definition of generic stability we use is that of~\cite[Definition~1.6]{gstabstab}.

It is well-known (see~\cite[Lemma~12.10]{poizat}) that every partial type which is finitely satisfiable in $M$ extends to a global type still finitely satisfiable in $M$, and that  if $p\in S(\monster)$ is finitely satisfiable in $M$ then $p$ is $M$-invariant (see~\cite[Theorem~12.13]{poizat}). Moreover all the notions above are monotone: for  instance if $p$ is generically stable over $A$ and $A\subseteq B$, then $p$ is generically stable over $B$, as Morley sequences over $B$ are in particular Morley sequences over $A$.

\begin{fact}[\!\!{\cite[Proposition~1(ii)]{piltan}}]\label{fact:gsfs}
  If $p$ is generically stable over a model $M$, then $p$ is finitely satisfiable in $M$.
\end{fact}

\begin{lemma}\label{lemma:fspartial}
  Suppose $p\in\invtypes_x(\monster)$ is finitely satisfiable in $M$ and $r\in S_{xy}(M)$ is consistent with $p$. Then $p\cup r$ is finitely satisfiable in $M$.
\end{lemma}
\begin{proof}
  Pick any $\phi(x)\in p$ and $\rho(x,y)\in r$. As $p\cup r$ is consistent, we have $p\proves \exists y\; (\phi(x)\wedge\rho(x,y))$, and as $p$ is finitely satisfiable in $M$ there is $m^0\in M$ such that $\models \exists y\; (\phi(m^0)\wedge\rho(m^0,y))$. In particular, $\models \exists y\; \rho(m^0, y)$, and since $\rho(m^0, y)\in L(M)$ and $M$ is a model there is $m^1\in M$ such that $\models \rho(m^0, m^1)$, so $(m^0, m^1)\models \phi(x)\land \rho(x,y)$.
\end{proof}
We can now prove the main result of this section. Part~\ref{point:gstab} can be seen as a generalisation of~\cite[Proposition~3.6]{tanovic}; the missing step to formally call it a generalisation would be to know that for a regular type $p$  the equivalence $p\nwort q\sse p\domd q$ held. To the best of the author's knowledge, this is currently only known for strongly regular generically stable types, or under additional assumptions such as stability. See~\cite{tanovic} for the definitions of \emph{regularity}  and \emph{strong regularity} in this context, and the next subsection for $\wort$.
\begin{thm}\label{thm:propertiespreserved}
  Suppose $A$ is a small set such that $p_x,q_y\in\invtypes(\monster, A)$ and $r\in S_{pq}(A)$ is such that $p\cup r\proves q$. 
  \begin{enumerate}
  \item If $A=M$ is a model and $p$ is finitely satisfiable in $M$, then so is $q$.
  \item If $p$ is definable over $A$, then so is $q$.
      \item \label{point:gstab}  If $A=M$ is a model and $p$ is generically stable over $M$, then so is $q$.
  \end{enumerate}
\end{thm}
\begin{proof}\*
  
\bigcircled 1  Let $\psi(y)\in q$, and let by hypothesis and compactness $\phi(x)\in p$ and $\rho(x,y)\in r$ be such that $\models \forall x,y\;\bigl((\phi(x)\land\rho(x,y))\implica\psi(y)\bigr)$. Lemma~\ref{lemma:fspartial} ensures the existence of $m^0, m^1\in M$ such that $\models \phi(m^0)\land \rho(m^0, m^1)$, and in particular $\models \psi(m^1)$.

\bigcircled 2  Work in $L(A)$. We want to show that for every $\psi(y;z^1)\in L(A)$ the set $d_q\psi\ssq S_{z^1}(A)$ is clopen; it is sufficient to show that $d_q\psi$ is open, as since $\psi$ is arbitrary then the complement $d_q(\neg \psi)$ of $d_q\psi$ will be open as well. Fix $d$ such that $q\proves \psi(y;d)$; we are going to find a formula $\delta(z^1)\in \tp(d/A)$ such that every element of $S_{z^1}(A)$ satisfying $\delta$ lies in $d_q\psi$, proving that  $\tp(d/A)$ is in the interior of $d_q\psi$.

Let $z\coloneqq z^0z^1$ and take $\phi(x;z^0)\in L(A)$, $e\in \monster$ and $\rho(x,y)\in r$  such that 
\[
  p\proves \underbrace{\phi(x;\underset{z^0}e)\land\forall y\;\Bigl(\bigl(\phi(x;\underset{z^0}e)\land\rho(x,y)\bigr)\implica\psi(y; \underset{z^1}d)\Bigr)}_{\eqqcolon \theta(x;ed)}
\]
 As $\theta(x;z)$ is an $L(A)$-formula and $p$ is definable over $A$,  the formula $\delta(z^1)\ceq (\exists z^0\;d_p\theta)(z^1)$  is as well over $A$. Suppose  $\tilde d\in \monster$ is such that $\models\delta(\tilde d)$, and let  $\tilde{e}\in \monster$ be such that  $\models d_p\theta(\tilde{e}, \tilde{d})$. By construction we have \[p\proves \phi(x,\tilde e)\land\forall y\;\Bigl(\bigl(\phi(x,\tilde e)\land\rho(x,y)\bigr)\implica\psi(y, \tilde d)\Bigr)\] and it follows that $p\cup \set\rho\proves \psi(y,\tilde d)$; therefore $\psi(y,\tilde d)\in q$.  As $\delta(z^1)\in \tp(d/A)$, we are done.

\bigcircled 3  Assume that $q$ is not generically stable over $M$, as witnessed by an $L(M)$-formula $\psi(y; w)$, some $\tilde d\in \monster^{\abs w}$, an ordinal $\alpha$ and a Morley sequence $\seq{\tilde b^i\mid i<\alpha}$ of $q$ over $M$ such that both  $I\coloneqq\set{i<\alpha\mid \models\neg \psi(\tilde b^i; \tilde d)}$ and $\alpha\setminus I$ are infinite and $\psi(y;\tilde d)\in q(y)$.

 By Fact~\ref{fact:gsfs} and Lemma~\ref{lemma:fspartial} $p\cup r$ is finitely satisfiable in $M$. Since $p\cup r\proves q$, the partial type  $p\cup r\cup q$ is finitely satisfiable in $M$ as well, and therefore extends to some $\hat r\in S(\monster)$  which is, again, finitely satisfiable in $M$, and in particular  $M$-invariant; take a Morley sequence $\seq{(a^i, b^i)\mid i\in I}$ of $\hat r$
  over $M$, let $f\in \aut(\monster/M)$ be such that $f(\seq{\tilde b^i\mid i\in I})=\seq{b^i\mid i\in I}$, and set $d\coloneqq f(\tilde d)$. Note that $p,q,r$ and $\psi(y;w)$ are fixed by $f$.
    
    Now let $J$ be a copy of $\omega$ disjoint from $I$ and let  $\seq{a^j\mid j\in J}$ realise a Morley sequence of $p$ over $Md\set{a^i\mid i\in I}$. We want to show that the concatenation of $\seq{a^i\mid i\in I}$ with $\seq{a^j\mid j\in J}$ contradicts generic stability of $p$ over $M$. By construction this is a Morley sequence over $M$, and if we find  $\chi(x;d)$ such that $\models \chi(a^i;d)$ holds for $i\in J$ but for no $i\in I$ then we are done, since  $I$ and $J$ are infinite.

As $\psi(y;d)\in q$ by $M$-invariance of $q$, there is by hypothesis $\phi(x,y)\in r$ such that $p(x)\proves \forall y\; \bigl(\phi(x,y)\implica \psi(y;d)\bigr)$. Let $\chi(x;d)$ be the last formula. By hypothesis, for $i\in J$ we have $\models \chi(a^i; d)$. On the other hand, for $i\in I$  we have  $(a^i,b^i)\models\phi(a^i, b^i)\land\neg \psi(b^i; d)$, and in particular  for all $i\in I$ we have $\models\neg \chi(a^i;d)$.
\end{proof}
\begin{rem}\label{rem:dfdgsbst}
  We are assuming that $p,q$ are $A$-invariant. It is \emph{not} true that if $p$ is finitely satisfiable/definable/generically stable in/over some $B\subseteq A$ then  $q$ must as well be such, for the same $B$. Even when $B=N\prec M=A$ are models, a counterexample can easily be obtained by taking $q$ to be the realised type of a point in $M\setminus N$. 
   \end{rem}
   \begin{question}\label{q:dwc}
     Is it true that in the setting of Remark~\ref{rem:dfdgsbst}  $q$ is domination-equivalent to a type finitely satisfiable/definable/generically stable in/over $N$?
   \end{question}
\begin{co}\label{co:invbarteqchange}
There is a theory $T$ where $\invtilde$ changes when passing to $T^\eq$.
\end{co}
\begin{proof}
  As generic stability is preserved by domination, this happens in any theory where $T$ does not have any nonrealised generically stable type but $T^\eq$ does, as such a type cannot be domination-equivalent to any type with all variables in the home sort. An example of such a theory is that of a structure $(M, <, E)$ where $(M,<)\models \textsf{DLO}$ and $E$ is an equivalence relation with infinitely many classes, all of which are dense.  
\end{proof}

  Such a thing cannot happen when passing from a \emph{stable} $T$ to $T^\eq$; see Remark~\ref{rem:tteqstab}.

  \begin{pr}\label{pr:gstabcommute}
  Generically stable types commute with every invariant type.
\end{pr}
\begin{proof}
The proof of~\cite[Proposition~2.33]{simon} goes through even without assuming \textsf{NIP} provided the definition of  ``generically stable'' is the one above.  
\end{proof}

Even if $(\invtilde, \otimes)$ need not be well-defined in general, a smaller object is.
  \begin{defin}
    Let $\gsinvtilde$ be the quotient by $\domeq$ of the space of types which are products of generically stable types.
  \end{defin}
  \begin{co}\label{co:gsinvtilde}
    $(\gsinvtilde, \otimes, \doms)$  is a well-defined, commutative ordered monoid.
  \end{co}
  \begin{proof}
    It follows immediately from Lemma~\ref{lemma:atleastontheleft} and Proposition~\ref{pr:gstabcommute} that, when restricting to the set of products of generically stable types, $\domeq$  is a congruence with respect to $\otimes$. As the generators of $\gsinvtilde$ commute, so does every pair of elements from it.
  \end{proof}

The reason we defined  $\gsinvtilde$ as  above is that  generic stability is not preserved under products: the type $p$ in~\cite[Example~1.7]{gstabstab} is generically stable but $p\otimes p$ is not.

  $\gsinvtilde$ may be significantly smaller than $\invtilde$, and even be reduced to a single point; this happens for instance in the Random Graph, or in \textsf{DLO}.

\subsection{Weak Orthogonality}

Another property preserved by domination is weak orthogonality to a type. This generalises (by Proposition~\ref{pr:sameasclassic}) a classical result in stability theory, see e.g.~\cite[Proposition~C.13'{}'{}'(iii)]{makkai}.

\begin{defin}\label{defin:wort}
  We say that $p\in S_x(\monster)$ and $q\in S_y(\monster)$ are \emph{weakly orthogonal}, and write $p\wort q$, iff $p\cup q$ is a complete global type. 
\end{defin}
Note that if $p$ is invariant then  $p\wort q$ is equivalent to $p\cup q\proves p\otimes q$, or in other words to the fact that for any $c\models q$ in some $\monster_1\satext \monster$ we have  $p\proves p\invext \monster c$.

In the literature  the name \emph{orthogonality} is sometimes (e.g.~\cite[p.~136]{simon} or~\cite[p.~310]{tanovic}) used to refer to the restriction of weak orthogonality to global invariant types. We will not adopt this convention here.

\begin{pr}\label{pr:wortpreserved}
  Suppose that $p_0, p_1\in \invtypes(\monster)$ are such that $p_0\doms p_1$ and $p_0\wort q$. Then $p_1\wort q$.
\end{pr}
\begin{proof}
  Fix $\monster_1\satext \monster$, work in  its elementary diagram and  suppose $p_0(x)\cup r(x,y)\proves p_1(y)$. We have to show that for any $c\in\monster_1$ realising $q$ we have  $p_1\proves p_1\invext \monster c$. By hypothesis, $p_0\proves p_0\invext \monster c$, and by Lemma~\ref{lemma:rsufficesforinvext} we have $(p_0\invext \monster c)\cup r \proves p_1\invext \monster c$, therefore $p_0 \cup r \proves p_1\invext \monster c$. This means that, for any $\psi(y,z)\in L(\monster)$  such that  $\psi(y,c)\in p_1\invext \monster c$, there are $\phi(x)\in p_0$ and  $\rho(x,y)\in r$ such that $\monster_1\models \forall x,y\; \bigl((\phi(x)\land\rho(x,y))\implica \psi(y,c)\bigr)$,  therefore \[\monster_1\models \forall y\; \bigl(
\bigl(    \exists x\;
    (\phi(x)\land\rho(x,y))\bigr)\implica \psi(y,c)
        \bigr)\]
  As $p_1(y)\cup r(x,y)$ is consistent, since it is satisfied by any realisation of $p_0(x)\cup r(x,y)$ by hypothesis, we have $p_1(y)\proves \exists x\; (\phi(x)\land\rho(x,y))$, and the conclusion follows.
\end{proof}
This entails the following slight generalisation of~\cite[Theorem~10.23]{poizat}.
\begin{co}\label{co:daptr}
  Let $p_x,q_y\in\invtypes(\monster)$. If $p\doms q$ and $p\wort q$, then $q$ is realised.
\end{co}
\begin{proof}
  From $p\doms q$ and $p\wort q$ the previous proposition gives $q\wort q$. But this can only happen if $q$ is realised, otherwise  $q(x)\cup q(y)\cup \set{x=y}$ and   $q(x)\cup q(y)\cup \set{x\ne y}$ are both consistent.
\end{proof}
\begin{rem}
  Tanovi\'c has proved in~\cite[Theorem~4.4]{tanovic} that if $p$ is strongly regular (see~\cite[Definition~2.2]{tanovic}) and generically stable then $p$ is $\le_{\mathrm{RK}}$-minimal among the nonrealised types, and for all invariant $q$ we have $p\nwort q\iff p\le_{\mathrm{RK}} q$. An immediate consequence of his result and of the previous corollary is that such types are also $\domd$-minimal among the nonrealised types.
\end{rem}

We conclude this section by remarking that a lot of properties are \emph{not} preserved by domination-equivalence, nor by equidominance. For instance, there is an $\omega$-stable theory with two equidominant types of different Morley rank, namely $T^\eq$ where $T$ is the theory of an equivalence relation with infinitely many classes, all of which are infinite. Another property that is not preserved is having the same dp-rank, a counterexample being \textsf{DLO}, where if $p$ is, say, the type at $+\infty$ we have $p\equidom p\otimes p$ even if the former has dp-rank $1$ and the latter has dp-rank $2$.

\section{Dependence on the Monster Model}\label{section:depmon}
In strongly minimal theories (see Example~\ref{eg:strminN}) $\invtilde\cong \mathbb N$ regardless of $\monster$ while in, say, the Random Graph, $\invtilde$ is very close to $\invtypes(\monster)$ by Proposition~\ref{pr:rgdegdom} and the subsequent discussion: the former is obtained from the latter by identifying types that only differ because of realised, duplicate, or permuted coordinates. It is natural to ask whether and how much the quotient $\invtilde$ depends on $\monster$, and the question makes sense even when $\otimes$ does not respect $\doms$.
This section investigates this matter.

\subsection{Theories with IP}
 The preorder $\doms$ is the result of a series of generalisations that began in~\cite{firstdefinition} with starting point the Rudin-Keisler order on ultrafilters. It is not surprising therefore that some classical arguments involving the latter object generalise as well.
We show in this subsection (Proposition~\ref{pr:ipinvtidlesize}) that, in the case of theories with \textsf{IP} (see~\cite[Chapter~2]{simon}), one of them is the  abundance of pairwise Rudin-Keisler inequivalent ultrafilters on $\mathbb N$; the classical proof goes through for $\domeq$ as well, and shows that even the cardinality of $\invtilde$ depends on $\monster$.

In this subsection $\class p$ stands for the $\domeq$-class of $p$. Even if we state everything for $\domeq$ and its quotient $\invtilde$, the same arguments work if we replace $\domeq$ by $\equidom$, $\invtilde$ by $\invbar$ and interpret $\class p$ as the class of $p$ modulo $\equidom$.

The following result is classical, see e.g.~\cite[Exercise~4(a) of Section~10.1 and Theorem~10.2.1]{hodges}.
\begin{fact}
  Let $T$ be any theory and $\lambda\ge \abs{T}$. Then $T$ has a $\lambda^+$-saturated and $\lambda^+$-strongly homogeneous model of cardinality at most $2^\lambda$.
\end{fact}
For the rest of this subsection, let  $\monster$ be $\lambda^+$-saturated and $\lambda^+$-strongly homogeneous of cardinality at most $2^{\lambda}$, let $\sigma$ be the least cardinal such that $\monster$ is not $\sigma^{+}$-saturated, and let $\kappa=\abs{\monster}$. Thus $\lambda^+\le \sigma\le \kappa\le 2^\lambda$. 
\begin{lemma}\label{lemma:sizeofclasses}
  In the notations above, for every $p\in \invtypes(\monster)$ we have $\abs{\class p}\le\abs{\set{q\mid q\domd p}}\le \ka^{<\sigma}$.
\end{lemma}
\begin{proof}
Clearly $\class p\subseteq \set{q\mid q\domd p}$.  For every $q\domd p$, there is some small $r_q$ such that $p\cup r_q\proves q$. If $r_q=r_{q'}$ then $q=q'$,  and therefore $\abs{\set{q\mid q\domd p}}$ is bounded by the number of small types. As ``small'' means of cardinality strictly less than $\sigma$, the number of such types is at most the size of $\bigcup_{A\subset \monster, \abs A<\sigma}S(A)$, which cannot exceed $\ka^{<\sigma}\cdot 2^{<\sigma}=\ka^{<\sigma}$.
\end{proof}

\begin{co}
  The same bound applies to sets of the form $\set{\class q\mid \class q\domd \class p}$, for a fixed $p$.
\end{co}

\begin{lemma}\label{lemma:ipkazs}
  If $T$ has \tf{IP}, then $2^\lambda=\kappa=2^{<\sigma}=\ka^{<\sigma}$.
\end{lemma}
\begin{proof}
  If $\phi(x;y)$ witnesses \tf{IP}, then over a suitable model of cardinality $\lambda$, which we may assume to be embedded in $\monster$, there are  $2^\lambda$-many $\phi$-types, and a fortiori types. This gives the first equality, and the same argument with any $\mu$ such that $\lambda\le \mu<\sigma$ gives the second one. The third one follows by cardinal arithmetic. 
\end{proof}
Recall the following property of theories with \textsf{IP}.
\begin{fact}\label{fact:manycoheirs}
  If $T$ has \tf{IP}, then for every $\lambda\ge\abs T$ there is a type $p$ over some $M\models T$ such that $\abs M=\lambda$ and $p$ has $2^{2^\lambda}$-many $M$-invariant extensions. Moreover, such extensions can be chosen to be over any $\lambda^+$-saturated model.
\end{fact}
\begin{proof}
This is  \cite[Theorem~12.28]{poizat}. The ``moreover'' part follows from the proof in the referenced source: in its  notation, it is enough to realise the $f$-types of the $b_w$ over $\set{a_\alpha\mid \alpha< \lambda}$.
\end{proof}

\begin{pr}\label{pr:ipinvtidlesize}
  If $T$ has \tf{IP} and $\monster$ is $\lambda^+$-saturated and $\lambda^+$-strongly homogeneous of cardinality $2^\lambda$, then $\invtilde$ has size $2^{\abs\monster}$.
\end{pr}
\begin{proof}
Since $\lambda^+$-saturation implies $\lambda^+$-universality, we may assume that the $M$  given by Fact~\ref{fact:manycoheirs} is an elementary submodel of $\monster$, and by the ``moreover'' part of Fact~\ref{fact:manycoheirs} we have $\abs{\invtypes(\monster)}\ge 2^{2^\lambda}$. But then by Lemma~\ref{lemma:sizeofclasses}
\[
2^\ka=    2^{2^\lambda}\le\abs{\invtypes(\monster)}=\sum_{\class p\in\invtilde}\abs{\class p}\le \abs{\invtilde}\cdot \ka^{<\sigma}
  \]
 Using Lemma~\ref{lemma:ipkazs} we obtain $2^{\ka}\le \abs{\invtilde}\cdot \ka$, and therefore $\abs{\invtilde}=2^{\ka}$.
\end{proof}
\begin{co}\label{co:ipdep}
  If $T$ has \tf{IP} then $\invtilde$ depends on $\monster$.
\end{co}
\begin{proof}
If $\monster_1$ is, say,   $\abs{\monster_0}^+$-saturated of cardinality $2^{\abs{\monster_0}}$, then $\abs{\operatorname{\widetilde{Inv}}(\monster_1)}=2^{2^{\abs{\monster_0}}}$.
\end{proof}
\begin{question}
  Is there an unstable \tf{NIP} theory where $\invtilde$ does not depend on $\monster$? Is there one where $\invbar$ does not depend on $\monster$?
\end{question}

\begin{question}
  Can $\invtilde$ or $\invbar$ be finite?
\end{question}
By the results above and Proposition~\ref{pr:ur1son}\footnote{\ldots and the fact that we only consider  theories with no finite models\ldots} it is enough to consider the \textsf{NIP} unstable case.

\subsection{The map $\mathfrak e$} \label{subsec:e}

Let $\monster_1 \satext \monster_0$. The map $p\mapsto p\invext \monster_1$ shows that, for every tuple of variables $x$, a copy of $\invtypes_x(\monster_0)$ sits inside    $\invtypes_x(\monster_1)$; for instance,  if $T$ is stable, this is nothing more than the classic identification of types over $\monster_0$ with types over $\monster_1$ that do not fork over $\monster_0$. 
\begin{defin}
If $\monster_0\smallprec \monster_1$, we define the map $\mathfrak e\from \widetilde{\operatorname{Inv}}(\monster_0)\to \widetilde{\operatorname{Inv}}(\monster_1)$ as $\mf e(\class{p})\coloneqq\class*{p\invext \monster_1}$.  
\end{defin}
\begin{pr}
The map   $\mf e$ is well-defined and weakly increasing. If moreover $\otimes$ respects $\doms$, then $\mathfrak e$ is also a homomorphism of monoids.
\end{pr}
\begin{proof}
 If $p\doms q$, as witnessed by $r$, by Lemma~\ref{lemma:rsufficesforinvext} we have $(p\invext \monster_1)\cup r\proves (q\invext \monster_1)$, and the first part follows.
  
Suppose now that $\otimes$  respects $\doms$ and denote for  brevity $p\invext \monster_1$ with $\widetilde p$.
Recall that, if $\phi(x,y)\in L(\monster_0)$, then $\phi\in p_x\otimes q_y$ if and only if for \emph{any} $b\models q$ we have $\phi(x,b)\in p\invext\monster_0b$. This in particular holds for any $b\models \widetilde q$ and shows that $(\widetilde p\otimes \widetilde q)\restr \monster_0=p\otimes q$, or in other words $(p\otimes q)\invext \monster_1=\tilde p\otimes \tilde q$. Therefore
 \[\mf e(\class{p})\otimes \mf e(\class{q})=\class{\widetilde p}\otimes \class{\widetilde q}=\class{(p\otimes q)\invext \monster_1}=\mf e(\class{p\otimes q})\] so $\mf e$ is a homomorphism of semigroups. As $\mf e$ clearly sends $\class 0$  to $\class 0$, because an extension of a realised type is realised, we have the conclusion.
\end{proof}

\begin{lemma}\label{lemma:suffinj}
Suppose that every time $p,q\in S(\monster_1)$ are $A_0$-invariant for some $A_0\smallsubset \monster_0$ and $p\doms q$ then this can be witnessed by some $r'\in S(A')$ such that $\monster_0$ is $\abs {A'}^+$-saturated and $\abs {A'}^+$-strongly homogeneous.\footnote{Note that $A'$ need not be a subset of $\monster_0$.} Then $\mf e$ is injective and $\mathfrak e(\class p)\doms \mathfrak e(\class q)$ implies $\class p \doms \class q$.
\end{lemma}
\begin{proof}
  We have to check that, in the previous notations, if $\widetilde p\doms \widetilde q$ then $p\doms q$. If $\widetilde p\doms \widetilde q$ can be witnessed by some $r$  with parameters in some $A\smallsubset \monster_0$, then we are done: by Lemma~\ref{lemma:babycraig} $p\cup r\proves q$.
  
As $\monster_0$ is $\abs{A_0\cup A'}^+$-saturated and  $\abs{A_0\cup A'}^+$-strongly homogeneous, up to taking unions we may assume $A'\supseteq A_0$, and by hypothesis we can find an $A_0$-isomorphic copy $A$ of $A'$ inside $\monster_0$. Let $f\in \aut(\monster_1/A_0)$ be such that $A=f(A')$ and define \[\monster_0'\coloneqq f^{-1}(\monster_0) \qquad p'\coloneqq f^{-1}(p)\in S(\monster_0') \qquad q'\coloneqq f^{-1}(q)\in S(\monster_0')\] As $\widetilde p$ and $\widetilde q$ are $A_0$-invariant they are fixed by $f$, so $p'\ssq \widetilde p$ and $q'\ssq \widetilde q$; by Lemma~\ref{lemma:babycraig} we therefore have $p'\cup r'\proves q'$, and so  $r\coloneqq f(r')$ witnesses both $\widetilde p\doms \widetilde q$ and $p\doms q$.
\end{proof}

The hypotheses of the lemma are satisfied for instance if $T$ has degenerate domination, or if $T$ is stable by Corollary~\ref{co:stableinjective}. Note that, should $\mathfrak e$ fail to be injective, we could still in principle have two monster models $\monster_0$ and $\monster_1$ of different cardinalities such that $\abs{\operatorname{\widetilde{Inv}}(\monster_0)}=\abs{\operatorname{\widetilde{Inv}}(\monster_1)}$. For instance, even in a theory with \textsf{IP},  the results of the previous subsection do not prevent this from happening in the case where  $\abs{\monster_0}$ and $\abs{\monster_1}$ are, say, strongly inaccessible cardinals.

\begin{question}\label{question:downwardclosed}
Is the image of $\mathfrak e$ downward closed? More generally, if $p\in \invtypes(\monster)$ is dominated by some $M$-invariant type, is $p$ then domination-equivalent to some $M$-invariant type?
\end{question}
\noindent By standard results (see~\cite[Lemma~2.18]{simon} and~\cite[Fact~1.9(2)]{gstabstab}), if Question~\ref{question:downwardclosed} has a positive answer then so does    Question~\ref{q:dwc}; for instance,  any $M$-invariant type finitely satisfiable in some small $N$ is finitely satisfiable in $M$.

\section{Stable Theories}\label{section:stabletheories}

The domination preorder we defined generalises a notion from classical stability theory.  For the sake of completeness, we collect in  this  section what is already known in the stable case.   From now on, we will assume some knowledge of stability theory from the reader, and $T$ will be stable unless otherwise stated; we  repeat this assumption for emphasis. References for almost everything that follows can be found in e.g.~\cite{buechler, gstheory, poizat}\footnote{In~\cite{buechler}, some results are only stated for theories with regular $\kappa(T)$; the reason for this is that~\cite{buechler} defines an a-model to be a strongly $\kappa(T)$-saturated model, as opposed to a strongly $\kappa_\mathrm{r}(T)$-saturated one.}. In this section, we mention \emph{orthogonality} of types, denoted by $\perp$, which is a strengthening of weak orthogonality that can be defined in a  stable theory for stationary types (see~\cite[Section~1.4.3]{gstheory}). For global types,  it coincides with weak orthogonality.

\subsection{The Classical Definition}
In the following definition $A$ is allowed to be a large set, e.g.\ we allow $A=\monster$.
\begin{defin}\label{defin:cldom}
We say that  $a$ \emph{weakly dominates $b$ over $A$} iff for all $d$ we have  $a\find A d\then b\find A d$. We say that  \emph{$a$ dominates $b$ over $A$}, written $a\cldoms_Ab$, iff for every $B\supseteq A$ if $ab\forkindep_AB$ then $a$ weakly dominates $b$ over $B$.
\end{defin}

\begin{fact}[{see~\cite[Lemma~1.4.3.4]{gstheory} and~\cite[Lemma~19.18]{poizat}}]
 Suppose  $A\subseteq B$ and $ab\find AB$. Then $a\cldoms_B b$ if and only if $a\cldoms_A b$. Moreover, over a $\abs T^+$-saturated model domination and weak domination are equivalent.
\end{fact} 

\begin{defin}\label{defin:cldomtypes}
For stationary $p,q\in S(A)$ we say that  $p\cldoms q$ iff there are $a\models p$ and $b\models q$ such that $a\cldoms_A b$. If $p\cldoms q\cldoms p$ we write $p\cldomeq q$. If there are $a\models p$ and $b\models q$ such that $a\cldoms_A b\cldoms_A a$ we  write $p\doteq q$.
\end{defin}

\begin{pr}\label{pr:sameasclassic}
Suppose that $T$ is stable and $p,q$ are global types. Then
\begin{enumerate}
\item $p\doms q$ if and only if  $p\cldoms q$.
\item $p\equidom q$ if and only if $p\doteq q$.
\item \label{point:bddsize}If $p\doms q$ then this is witnessed by some $r\in S_{pq}(M)$ with $\abs{M}\le \abs{T}$.
\end{enumerate}
\end{pr}

  \begin{proof}[Proof Sketch]
 If $p\cup r\proves q$, where $r\in S(M)$, and $(a,b)\models r$, then it can be checked that $a$ weakly dominates $b$ over $M$, and if $M$ is large enough this yields $p\cldoms q$. In the other direction, take $a\models p$, $b\models q$ witnessing domination and consider their type over some $M$  of size at most $\abs T$ such that $ab\find M\monster$. The rest follows easily. For more details, see e.g.~\cite[Lemma~1.4.3.4~(iii)]{gstheory}.
\end{proof}
 More conceptual proofs of  the  first and last point can be obtained from the classical results that $p\cldoms q$ if and only if $q$ is realised in the prime a-model containing a realisation of $p$, and that prime a-models are a-atomic (see~\cite[Lemma~1.4.2.4]{gstheory}). Note that a consequence of this equivalence is that in a stable theory semi-a-isolation (i.e.\ $\doms$ by point~\ref{point:bddsize} of the previous Proposition) is the same as a-isolation: if $p\cup r\proves q$ then $r$ can be chosen such that $p\cup r$ is complete, despite $r$ being small.

\begin{co}\label{co:stableinjective}
  If $T$ is stable,   then $\mf e$ is injective  and $\mathfrak e(\class p)\doms \mathfrak e(\class q)$ implies $\class p \doms \class q$.
\end{co}
\begin{proof}
  By point~\ref{point:bddsize} of Proposition~\ref{pr:sameasclassic} we can apply Lemma~\ref{lemma:suffinj}.
  \end{proof}

\begin{rem}\label{rem:tteqstab}
  While studying $\invtilde$  in a stable $T$, there is no harm in passing to $T^\eq$, which we see as a multi-sorted structure, for the following reason. Even without assuming stability, every type $p\in S(\monster)$ in $T^\eq$ is dominated by, and in particular  (if it is nonrealised) not weakly orthogonal to, a type $q\in S(\monster)$ with all variables in the home sort via the projection map. Suppose now that $T$ is stable and let $M$ be such that $p$ and $q$ do not fork over $M$. By (the proof of)~\cite[Lemma~19.21]{poizat} there is a (possibly forking) extension of $q\restr M$ which is equidominant with $p$. Trivially, this extension has all variables in the home sort. We would like to thank Anand Pillay for pointing this out.
  \end{rem}
  \begin{rem}\label{rem:simplecase}
    Definition~\ref{defin:cldom} makes sense also in simple theories, and more generally in rosy theories if we replace forking by \th -forking (see~\cite{thorndom}). One can then give a definition of $\cldoms$ even for types that are not stationary but, in the unstable case, even for global types the relation  $\cldoms$ need not coincide with $\doms$. For instance, in the notation of Definition~\ref{defin:tfc}, let $(b,c)\models q_1$ and $a\models p\invext \monster bc$, and recall that in $T$ forking is characterised as
    \[
      e\find C d\iff (e\cap d\subseteq C) \land (\pi e\cap \pi d\subseteq \pi C)
    \]
    It follows that, for all $B\supseteq \monster$ such that $abc\find \monster B$, and for all  $d$ such that $ab\find B d$, we have  $abc\find B d$, and therefore $ab\cldoms_\monster abc$. Since $\tp(a,b/\monster)=p\otimes q_0$ and $\tp(a,bc/\monster)=p\otimes q_1$ this shows $p\otimes q_0\cldoms p\otimes q_1$, but by Proposition~\ref{pr:dombreaks} $p\otimes q_0\ndoms p\otimes q_1$. 
  \end{rem}

\subsection{Thin Theories}

Recall that a stable theory is \emph{thin}  iff every complete type has finite weight (see~\cite[Section~1.4.4]{gstheory} for the definition of weight). For instance  superstable theories are thin (\!\!\cite[Corollary~1.4.5.8]{gstheory})  and so are theories with no dense forking chains (\!\!\cite[Lemma~4.3.7]{gstheory}) or where every complete type has \emph{rudimentarily finite} weight (\!\!\cite[Proposition~4.3.10]{gstheory}). This hypothesis provides a structure theorem for $\invtilde$, namely Theorem~\ref{thm:bigoplusN}. This result is implicit in the literature (see~\cite[Proposition~4.3.10]{gstheory}), but we need to state it is as done below for later use.
\begin{fact}[\!\!{\cite[Lemma~1.4.4.2]{gstheory}}]\label{fact:w1pod}
  If $p$ and $q$  have both weight $1$ then the following are equivalent:
  \begin{enumerate*}
  \item $p\centernot\perp q$.
  \item $p\domeq q$.
  \item $p\equidom q$.
  \end{enumerate*}
\end{fact}
\begin{fact}[{\!\!\cite[Lemma~5.6.4 (iv)]{buechler}}]\label{fact:weightdom}
  Weight is preserved by domination-equivalence.
\end{fact}

\begin{lemma}\label{lemma:copyofN}
  If $p$ has weight $w(p)=1$, then the monoid generated by $\class p$ in $\invtilde$ is isomorphic to $\mathbb N$.
\end{lemma}
\begin{proof}
Since weight is additive over $\otimes$ (\!\!\cite[Proposition~5.6.5 (ii)]{buechler}) we have $w(p^{(n)})=n$ and we conclude by Fact~\ref{fact:weightdom} that the map $n\mapsto \class{\pow pn}$ is an isomorphism between  $\mathbb N$ and  the monoid generated by $\class p$.
\end{proof}
\begin{thm}\label{thm:bigoplusN}
  If $T$ is thin, then there are a cardinal $\kappa$, possibly depending on $\monster$, and an isomorphism  $f\from\invtilde\to \bigoplus_\kappa \mb N$.  Moreover, $p\perp q$ if and only if  $f(p)$ and $f(q)$ have disjoint supports.
\end{thm}
\begin{proof}
  Let $\seq{\class{p_i}\mid i<\kappa}$ be an enumeration without repetitions of the $\domeq$-classes of types of weight $1$.  For such classes, define $f(\class{p_i})$ to be the characteristic function of $\set i$, then extend $f$ to classes of products of weight-one types by sending $\class{p\otimes q}$ to $f(\class p)+f(\class q)$ and $\class 0$ to the function which is constantly $0$. It is easy to show using Fact~\ref{fact:w1pod} and Corollary~\ref{co:daptr} that $f$ is well-defined, i.e.\ does not depend on the decomposition as product of weight-one types, and that $f$ is injective.
  By~\cite[Proposition~4.3.10]{gstheory} in a thin theory every type  is domination-equivalent to a finite product  of weight-one types, so $f$ is defined on the whole of $\invtilde$.   By Lemma~\ref{lemma:copyofN} if $w(p)=1$ then the monoid generated by $\class p$ is isomorphic to $\mathbb N$ and this easily entails that $f$ is surjective.  It is also clear that $f$ is an isomorphism of ordered monoids. Since  two types of weight $1$ are either weakly orthogonal or domination-equivalent by Fact~\ref{fact:w1pod} and, by~\cite[Proposition~C.5(i)]{makkai}, in stable theories $p\perp q_0\otimes q_1$ if and only if $p\perp q_0$ and $p\perp q_1$, the last statement follows.
\end{proof}

\begin{rem}\label{rem:readoffweight}
Weight, which is preserved by domination-equivalence (Fact~\ref{fact:weightdom}), can, in the thin case,  be read off $f(\invtilde)$ by taking ``norms''. Specifically, if $f(\class p)=(n_i)_{i<\kappa}$, then $w(p)=\sum_{i<\kappa} n_i$  (recall that every  $(n_i)_{i<\kappa}\in \bigoplus_\kappa \mathbb N$ has finite support).  
\end{rem}

\begin{pr}
  If $T$ is thin, then $\equidom$ and $\domeq$ coincide.
\end{pr}
\begin{proof}
By~\cite[Theorem~4.4.10]{kim} every type is in fact \emph{equidominant} with a finite product of types of weight $1$. The conclusion then follows from Fact~\ref{fact:w1pod} and the fact that, as $T$ is stable, $\otimes$ respects both $\domeq$ and $\equidom$.
\end{proof}

\subsection{Dimensionality and Dependence on the Monster}
At least in the thin case 
some classical results imply that independence of $\invtilde$ from the choice of $\monster$ is equivalent to \emph{dimensionality} of $T$, also called \emph{non-multidimensionality}. 
\begin{defin}
Let $T$ be stable.  We say that $T$ is \emph{dimensional} iff for every nonrealised global type $p$ there is a global type $q$ that does not fork over $\emptyset$ and such that $p\centernot\perp q$. We say that $T$ is \emph{bounded} iff $\abs{\invtilde}<\abs{\monster}$.
\end{defin}

If $T$ is thin, then $T$ is dimensional if and only if it is bounded, as follows e.g.~from Theorem~\ref{thm:bigoplusN} (alternatively, see the proof of~\cite[Lemma~7.1.2]{buechler}, but replace ``superstable'' with ``thin'' and ``regular types'' with ``weight-one types''). In this case  the number of copies of $\mathbb N$ required is  bounded by $2^{\abs T}$, and by $\abs T$ if $T$ is totally transcendental, see e.g.~\cite[Corollary~7.1.1]{buechler}. In fact, some sources define boundedness only for superstable theories, essentially as boundedness of the number of copies of $\mathbb N$ given by Theorem~\ref{thm:bigoplusN}.

\begin{conjecture}\label{conjecture:nmd}
Let $T$ be stable.  The following are equivalent: 
  \begin{enumerate*}
  \item \label{point:bounded} $T$ is bounded.
  \item $T$ is dimensional.
  \item \label{point:surj}$\mathfrak e$ is surjective.
  \end{enumerate*}
\end{conjecture}

$1\allora 2$  follows from~\cite[Proposition~5.6.2]{buechler} and $3\allora 1$ is trivial, so it remains to prove $2\allora 3$, namely that if there is a type over $\monster_1$ not domination-equivalent to any type that does not fork over $\monster_0$, then there is a type orthogonal to every type that does not fork over $\emptyset$.

\begin{pr}
  If $T$ is thin  then Conjecture~\ref{conjecture:nmd} holds.
\end{pr}
\begin{proof}
  Suppose $\monster_0\smallprec \monster_1$ and let $f_j\from  \operatorname{\widetilde{Inv}}(\monster_j)\to \bigoplus_{\kappa_j} \mathbb N$, for $j\in \set{0,1}$, be given by Theorem~\ref{thm:bigoplusN}. Let \[g\coloneqq f_1\circ \mathfrak e \circ f_0\inverse\from\bigoplus_{\kappa_0} \mathbb N\to \bigoplus_{\kappa_1} \mathbb N\]
Since weight is preserved by nonforking extensions (e.g.\ by~\cite[Definition~5.6.6~(iii)]{buechler}),  $\mathfrak e$  sends types of weight $1$ to types of weight $1$. Therefore  by Remark~\ref{rem:readoffweight}  we may decompose the codomain of $g$ as \[\bigoplus_{\kappa_1} \mathbb N\cong \bigoplus_{i<\kappa_0}\mathbb N\oplus \bigoplus_{\kappa_0\le i<\kappa_1}\mathbb N\] where the direct summand $\bigoplus_{i<\kappa_0}\mathbb N$ may be assumed to coincide with $\im g$. It then follows that if $\mathfrak e$ is not surjective then we can find $\class p\notin \im \mathfrak e$   such that $p$ has  weight $1$. Again by Theorem~\ref{thm:bigoplusN}, such a $p$ needs to be orthogonal to every type in the union of $\im \mathfrak  e$, which is the set of  types that do not fork over $\monster_0$. In particular, $p$ is orthogonal to every type that does not fork over $\emptyset$.
\end{proof}

  A possible attack in the general case could be, assuming $\mathfrak e$ is not surjective, to try to find a type of weight $1$ outside of its  image. This will be either orthogonal to every type that does not fork over $\monster$, or dominated by one of them by~\cite[Corollary~5.6.5]{buechler}. If we knew a positive answer to Question~\ref{question:downwardclosed} at least in the stable case, and if we managed to find a type as above, then  we would be done. 

  A possibly related notion is  the \emph{strong compulsion property} (see~\cite[Definition~2]{hyttinen}); it implies that every type over $\monster_1\satext \monster_0$ is either orthogonal to $\monster_0$ or dominates a type that does not fork over it.  Whether all countable stable $T^\eq$ have a weakening of this property is~\cite[Conjecture~18]{hyttinen}.

  We conclude  with two easy consequences of  some classical results.

\begin{defin}
  A stable theory $T$ is \emph{unidimensional} iff whenever $p\perp q$ at least one between $p$ and $q$ is algebraic.
\end{defin}
If $T$ is totally transcendental then unidimensionality is the same as categoricity in every cardinality strictly greater than $\abs T$ (see~\cite[Proposition~7.1.1]{buechler}). Unidimensional theories may still fail to be totally transcendental, e.g.\ $\Th(\mathbb Z,+)$ is such. Anyway, the following  classical theorem by Hrushovski~(see \cite[Theorem~4]{unidss}) tells us that the situation cannot be much worse than that. 
\begin{thm}[Hrushovski]
Every unidimensional theory is superstable.
\end{thm}
\begin{co}\label{co:unidim}
A stable  $T$ is unidimensional if and only if $\invtilde\cong\mb N$.
\end{co}
\begin{proof}
If $T$ is unidimensional, by Hrushovski's result we have the hypothesis of Theorem~\ref{thm:bigoplusN}, and the conclusion then follows easily from the definition of unidimensionality. In the other direction, the hypothesis yields that any two types are $\doms$-comparable, but if $p\wort q$ and $p\doms q$ then $q$ is realised by Corollary~\ref{co:daptr}.
\end{proof}

Compare the previous corollary with~\cite[Proposition~5]{firstdefinition}. Note that the hypothesis that $T$ is stable is necessary:  in the random graph if $p\wort q$ then one between $p$ and $q$ must be algebraic, but $\invtilde$ is not commutative by Corollary~\ref{co:rgncomm}.

  \begin{pr}\label{pr:ur1son}
  If $T$ is stable then $\mathbb N$ embeds in $\invtilde$. 
\end{pr}
\begin{proof}
  By~\cite[Lemma~13.3 and p.~336]{poizat} in any stable theory there is always a type $p$ of U-rank $1$, and in particular of weight  $w(p)=1$ (see~\cite[before Theorem~19.9]{poizat}). The conclusion follows from Lemma~\ref{lemma:copyofN}.  
\end{proof}

\end{document}